\documentclass[11pt,oneside,english,reqno]{article}
\usepackage[utf8]{inputenc}
\usepackage[T1]{fontenc}
\usepackage[bitstream-charter]{mathdesign}
\usepackage[authoryear]{natbib} 
\usepackage{fullpage}
\usepackage{amsmath}
\usepackage[amsthm,thmmarks]{ntheorem} 
\usepackage{jmueller}
\usepackage[mathscr]{euscript}
\usepackage{accents}
\usepackage[right]{eurosym}
\usepackage{graphicx}
\usepackage{xy}\xyoption{all}
\usepackage{threeparttable}
\usepackage[english]{babel}
\usepackage{url}
\usepackage{fixltx2e}[2006/03/24]
\usepackage{color}
\usepackage{algorithmic}
\usepackage{float}
\usepackage{ifpdf}
\usepackage[pdfusetitle,colorlinks=true]{hyperref}
\ifpdf 
  \pdfoutput=1
\fi

\newcommand{\mycommand}[2]{\ifdefined#1\renewcommand{#1}{#2}\else\newcommand{#1}{#2}\fi}

\newcommand{\comment}[1]{}

\newcommand{\N}{\ensuremath{\mathbb N}}

\newcommand{\Z}{\ensuremath{\mathbb Z}}

\newcommand{\R}{\mathbb R}
\newcommand{\Rn}{\ensuremath{\mathbb R^{n}}}
\newcommand{\Rm}{\ensuremath{\mathbb R^{m}}}

\newcommand{\bin}{\{0,1\}}

\mycommand{\mit}{\text{ with }}
\newcommand{\und}{\text{ and }}

\newcommand{\fuer}{\text{ for }}

\newcommand\inv{^{-1}}

\newcommand{\valp}{\p^{*}}

\newcommand{\val}[1]{#1^*}
\mycommand{\valdelta}{\delta^*}
\mycommand{\Hat}{{\ensuremath{H_{a,t}}}}
\mycommand{\Hp}{{H_{a,t}^+}}
\mycommand{\Hm}{{H_{a,t}^-}}

\newcommand{\hours}{{\ensuremath{t\in T}}}

\mycommand{\d}{\ensuremath{\delta}}
\newcommand{\p}{\ensuremath{\pi}}

\mycommand{\poly}{\ensuremath{\mathcal{P}}}

\ifdefined\Satz\renewcommand{\Satz}[1]{Satz \ref{satz:#1}}\else\newcommand{\Satz}[1]{Satz \ref{satz:#1}}\fi
\newcommand{\tp}{^{\top}}

\newcommand{\vallambda}{\lambda^*}

\newcommand{\convex}{_{C}}
\newcommand{\nonconvex}{_{N}}

\floatstyle{ruled}
\newfloat{algorithm}{thp}{lop}
\floatname{algorithm}{Algorithm}
\numberwithin{algorithm}{section}

\begin{document}

\title{Competitive Equilibrium Relaxations in General Auctions}

\author{Johannes C. M\"uller$^{\rm a}$\vspace{6pt}\\
{\fontsize{10}{13}\selectfont$^{\rm a}$\it
Department of Mathematics,
FAU Erlangen-N\"urnberg,
Cauerstr. 11, D-91058 Erlangen, Germany;}\\
{\fontsize{10}{13}\selectfont\it%
Email: Johannes.Mueller@fau.de;}
}

\ifdefined\hypersetup
  \hypersetup{
    pdfauthor={Johannes C. M\"uller},
    pdfkeywords={game theory, auctions/bidding, integer programming,
       nonlinear programming, MPEC},
    linkcolor=black,
    citecolor=black,
    filecolor=black,
    urlcolor=black,
  }
\fi

\date{\fontsize{8}{13}\selectfont\it(September 3, 2013)}

\maketitle

\begin{abstract}
The goal of an auction is to determine commodity prices
such that all participants are perfectly happy.
Such a solution is called a competitive equilibrium
and does not exist in general. For this reason we are interested 
in solutions which are similar to a competitive equilibrium.

The article introduces two relaxations of a competitive equilibrium for general auctions.
Both relaxations determine one price per commodity 
by solving a difficult non-convex optimization problem.
The first model is a mathematical program with equilibrium constraints (MPEC),
which ensures that each participant is either perfectly happy or his bid is rejected.
An exact algorithm and a heuristic are provided for this model. 
The second model is a relaxation of the first one and 
only ensures that no participant incurs a loss.
In an optimal solution to the second model, no participant can be made 
better off without making another one worse off.
\end{abstract}

\section*{Outline}
Section \ref{sec introduction} provides a brief introduction to auctions, bid expression,
strict linear pricing schedules, and surplus maximization.
Section \ref{sec definitions} provides definitions for quasi-linear utility, 
competitive equilibria, and extends these concepts by introducing decision sets,
quantity functions, and decision valuations. With these extensions, it becomes
convenient to apply the definitions to real-world optimization problems.
Section \ref{sec competitive equilibria} introduces the fundamental welfare theorem:
under certain convexity assumptions, the surplus maximizing solution admits a strict 
linear pricing schedule, at which the individual optimization problems of each participant
are maximized. Difficulties arise when some of the bids are non-convex and the
auction determines exactly one price per tradable commodity.
Section \ref{sec competitive relaxations} addresses these difficulties and
introduces two relaxations of a competitive equilibrium.
Section \ref{sec maximization or rejection} proposes a model and an exact algorithm that 
maximizes the economic surplus such that the individual optimization problem of a bid is 
either maximized or the bid is rejected.
Section \ref{sec non-negative surplus} introduces a model that maximizes the economic 
surplus such that no participant incurs a loss. The model computes an efficient solution,
that is, the surplus of one participant cannot be increased without
decreasing the surplus of another participant.
In Section \ref{sec real-world application} we refer to a large scale real-world
application.

\section{Introduction}
\label{sec introduction}

In this chapter we will provide a general definition of an auction.
An auction is coordinated by an \emph{auctioneer}. All auctions have in common 
that there is a non-empty set of \emph{buyers} and a non-empty set of 
\emph{sellers}. Otherwise the outcome of the auction is trivial, as nothing
can be traded. Furthermore, there is a non-empty set of different tradable 
\emph{commodities} and there might exist several interchangeable copies of
each commodity. Examples for commodities are company shares,
futures contracts, electricity at a specific location and time, and so on.
In general it is difficult to distinguish between buyers and sellers,
as there might exist participants who just want to swap different commodities.
Such a participant is a buyer and a seller at once.
For this reason we will use the term \emph{participants}
(also called \emph{bidders}) instead of buyers and sellers.

\subsection{Bid Expression}

The participants need to tell the auctioneer in which commodity bundles
they are interested. In our auction each participant submits
a \emph{bid} that represents his interests. A bid is determined by
the following \emph{bid parameters}:

\begin{defn}
Let $m$ be the number of tradable commodities. The bid of a participant
$i$ is determined by his \emph{bid parameters} $(D_i,v_i,f_i)$, where
\begin{itemize}
 \item $D_i$ is the feasible region of his decision variables
       (hereinafter also referred to as \emph{decision set}),
 \item $v_i:D_i\rightarrow\R$ is his \emph{decision valuation} function
       (benefit if $>0$/ costs if $<0$), and
 \item $f_i:D_i\rightarrow\R^m$ is his \emph{quantity function}.
\end{itemize}
\end{defn}

In Subsection \ref{subsec surplus maximization} 
we will see that these parameters describe individual optimization problems
that return the optimal demanded or supplied
quantities depending on the commodity prices given by the auctioneer.
In \cite{deVries2003} these individual optimization problems are called \emph{oracles}.

Let $(D_i,v_i,f_i)$ be the bid parameters of a participant $i$. If $\delta_i\in D_i$, 
then $\delta_i$ is a feasible decision and $f_i(\delta_i)$ is the associated
quantity vector in $\R^m$. Positive values indicate that the participant
demands the specified amount of a good and negative values
indicate that he supplies the specified amount of a good.
The value $v_i(\delta_i)$ indicates the benefit (or cost) that is associated
with the decision $\delta_i$. If the quantity function $f_i$ is injective on $D_i$,
then it assigns a unique benefit (or cost) to each commodity
bundle in $f_i(D_i)$. However, we do not need to require injectivity.

\subsection{Clearing Condition}

Let $M$ be the set of commodities, $I$ be the set of participants, and
$(D_i,v_i,f_i)$ be the bid parameters of participants $i\in I$.
A solution to an auction must satisfy at least the following two constraints:
\begin{align}
&\sum_{i\in I}f_i(\delta_i)=0\label{clearing condition},\\
&\delta_i\in D_i&&\forall i\in I.
\end{align}
The first equation is called \emph{clearing condition} 
and ensures that for each commodity the bought quantity
minus the sold quantity is equal to zero. The equality sign is important as the
auctioneer is not interested in keeping any goods. In some auctions there is
only one seller and the seller is the auctioneer. In this case we assume 
that the seller and the auctioneer are different parties. The
seller will keep the commodities which are not sold.
The second equation ensures that the decision variables of each participant
are in the respective decision set.

\subsection{Linear Pricing Schedules}

Participants who supply a commodity will only participate in the auction
if they receive money for supplying a commodity. The money is to be collected
from the participants who demand the commodity. 
We will now define a common pricing schedule, which is a multidimensional 
extension of the definition in \cite[p. 136]{tirole1988}.
\begin{defn}
Let $m$ be the number of different commodities. A \emph{pricing schedule}
$T:\R^m\rightarrow\R$ is a map that returns the total amount of money 
to be paid by a participant depending on his consumption vector
$q\in\R^m$. A negative $T(q)$ specifies the amount of money to be received
and $-q_i$ bought units model $q_i$ sold units of commodity $i$.
A pricing schedule is called \emph{linear} if the map is linear,
i.e., $T(q)=\p\tp q$. In this case $\p$ is called a 
\emph{linear price vector} and $\p_i$ is the price per unit for commodity $i$.
\end{defn}
\begin{defn}
A pricing schedule is \emph{strict linear} if it is linear and
the number of commodities $m$ is equal to the number of clearing conditions
in the auction model (\cite{vanvyve2011}).
\end{defn}
An example for a linear pricing schedule which is not strict linear
can be found in \cite{oneill2005}. There, the number of commodities equals the 
number of clearing conditions plus the number of binary variables.

\subsection{Surplus Maximization}\label{subsec surplus maximization}
We assume that the auctioneer decides to implement a strict linear pricing
schedule.
\begin{defn}
Let $T$ be a strict linear pricing schedule and let $(D_i,v_i,f_i)$ be the bid
parameters of participant $i$. Then his \emph{surplus} depending 
on his decision variable $\delta_i\in D_i$ is given by
\[
 v_i(\delta_i)-T(f_i(\delta_i)).
\]
The \emph{individual optimization problem}, which maximizes his surplus, is given by
\begin{align*}
 \max\{v_i(\delta_i)-T(f_i(\delta_i))\mid \delta_i\in D\}.
\end{align*}
A participant is \emph{perfectly happy} if his individual optimization problem is
maximized. The sum of the surpluses of all participants is called 
\emph{economic surplus} or \emph{social welfare}.
\end{defn}
Note, that the surplus is non-negative whenever we have $v_i(\delta_i)\ge T(f_i(\delta_i))$.
In other words if the participant is a buyer, his surplus
is non-negative whenever the benefit is greater or equal
to the amount of money to be paid. If he is a seller, the surplus is non-negative
whenever $-T(f_i(\delta_i))\ge-v_i(\delta_i)$, i.e., whenever the amount of money to be received
is greater or equal to the costs.

Now assume that the auctioneer wants to maximize the economic surplus, subject to
the clearing condition and the feasibility of the decision variables:
\begin{align}
\max\quad&\sum_{i\in I}\left(  v_i(\delta_i)-\p\tp f_i(\delta_i) \right),
  \label{eq surplus max 1}\\
\st\quad
&\sum_{i\in I}f_i(\delta_i)=0,\label{eq surplus max 3}\\
&\delta_i\in D_i,&&\forall i\in I,\\
&\p\in\R^m.\label{eq surplus max 2}
\end{align}

\begin{prop}
The following optimization problem is equal to \eqref{eq surplus max 1}-\eqref{eq surplus max 2}.
\begin{align}
\max\quad&\sum_{i\in I}v_i(\delta_i),\label{eq competitive eq 1}\\
\st\quad
&\sum_{i\in I}f_i(\delta_i)=0,\\
&\delta_i\in D_i,&&\forall i\in I,\\
&\p\in\R^m\label{eq competitive eq 2}.
\end{align}
\end{prop}
\begin{proof}
Let $(d,\p)$ be a feasible solution to \eqref{eq surplus max 1}-\eqref{eq surplus max 2}.
Equation \eqref{eq surplus max 3} yields that the objective is given by
\[
 \sum_{i\in I}\left(  v_i(\delta_i)-\p\tp f_i(\delta_i) \right)
= \sum_{i\in I} v_i(\delta_i) - \sum_{i\in I} \p\tp f_i(\delta_i)
= \sum_{i\in I} v_i(\delta_i) - \p\tp \sum_{i\in I} f_i(\delta_i)
= \sum_{i\in I} v_i(\delta_i).
\]
\end{proof}

Observe that $\p$ can be chosen arbitrarily. The economic surplus is not depending on
the values of $\pi$. If we choose arbitrary prices, then some of the participants
might \emph{incur a loss}, that is, they have a negative surplus.
Section \ref{sec definitions} introduces the competitive equilibrium,
a situation where all participants are perfectly happy. In particular no
participant incurs a loss. Section \ref{sec competitive equilibria} shows
that a competitive equilibrium exists if the model \eqref{eq competitive eq 1}%
-\eqref{eq competitive eq 2} is convex.
Section \ref{sec competitive relaxations} addresses the non-convex case.
There, an optimal solution to \eqref{eq competitive eq 1}-\eqref{eq competitive eq 2}
does not necessarily posses a strict linear pricing schedule where no participant
incurs a loss.

\section{Economic Definitions and Generalizations}
\label{sec definitions}
\def\s{^*}

\begin{defn}
Let $I$ be a finite set of auction participants
and let $L$ be a finite set of tradable commodities.
The first element in $L$ will be called \emph{numéraire commodity} and represents money.
$T:=L\setminus\{0\}$ is the set of tradable commodities without the numéraire commodity.
A set $X_i\subseteq\R^{L}$ is called a \emph{quantity set} of participant $i$
(also called \emph{consumption set} / \emph{production set}).
A positive entry in a quantity vector $x_i\in X_i$ indicates that participant $i$
receives the specified amount of that good, and 
a negative entry indicates that he gives away the specified amount of that good.
A function $u_i:X_i\rightarrow\R$ is called a \emph{utility function} of participant $i$.

If a utility function has the form $u_i(w,q)=w+\Phi_i(q)$ with $\Phi_i:\R^T\rightarrow\R$,
then it is called a \emph{quasi-linear utility function}. The function $\Phi_i$
is called \emph{valuation} (also called \emph{benefit} if $\ge0$ / \emph{cost} if $\le0$).

These definitions are consistent with \cite{mascolell1995} and \cite{blumrosenNisan2007}.
\end{defn}
\begin{prop}\label{prop:willingness}
Let $i$ be a consumer and let $u_i(w,q)=w+\Phi_i(q)$ be his quasi-linear utility function.
The \emph{willingness to pay} function $M_i:\R^T\rightarrow\R$ that returns the amount
of money the consumer is willing to pay for the consumption of the bundle $q\in\R^T$
is defined as
\[
u_i(w_0-M_i(q),q)=u_i(w_0,0).
\]
If $u_i$ is quasi-linear and $\Phi(0)=0$, then we have $M_i(q)=\Phi_i(q)$ for all $q\in\R^T$,
i.e., the willingness to pay is equal to the benefit.
\end{prop}

\begin{defn}\cite[Def. 10.B.1]{mascolell1995}
An \emph{allocation} $(x_i)_{i\in I}$ is a tuple of
quantity vectors $x_i\in X_i$ for each participant (buyer/seller) $i\in I$.
\end{defn}

\begin{defn}\cite[Def. 10.B.3]{mascolell1995}\label{def competitive}
An allocation $(x_i^*)_{i\in I}$ and a price vector $p^*\in\R^L$
constitute a \emph{competitive (Walrasian) equilibrium} with respect to
a \emph{strict linear pricing schedule} if the following conditions hold.
\begin{itemize}
\item Utility maximization for each participant:
\begin{equation}
\forall i\in I:\qquad x^*_i\in\arg\max\left\{u_i(x_i)\mid x_i\in X_i, (p^*)\tp x_i\le 0\right\}
\end{equation}
\item Clearing condition for each commodity:
\begin{equation}
\forall l\in L:\qquad \sum_{i\in I} x^*_{i,l} = 0
\end{equation}
\end{itemize}
The constraint $(p^*)\tp x_i\le 0$ is called the \emph{budget constraint}
of participant $i$.
Note that the clearing condition of money must hold in a competitive equilibrium
because it is also a commodity in $L$.
\end{defn}

\begin{prop}\label{prop:strict:competitive:equilibrium}
Let the utility functions of all participants $i\in I$ be quasi-linear,
i.e., $u_i(w,q)=w+\Phi_i(q)$.
Let the quantity sets of all participants $i\in I$ be given by $X_i=\R\times Q_i$
with $Q_i\subseteq\R^T$.
Let the price of money (the numéraire commodity) amount to $p^*_0=1$ currency unit.

Then the allocation $\left((w_i^*,q_i^*)\right)_{i\in I}$ and
the price vector $(1,p^*_T)\in\R^L$ constitute a \emph{competitive equilibrium}
with respect to a \emph{strict linear pricing schedule}
if and only if the following conditions hold.
\begin{itemize}
\item Utility maximization for each participant:
\begin{align}
&\forall i\in I:\qquad q^*_i\in\arg\max\left\{\Phi_i(q_i)-(p^*_T)\tp q_i \mid q_i\in Q_i\right\}
\label{eq:b3}\\
&\forall i\in I:\qquad w^*_i=-(p^*_T)\tp q^*_i\label{eq:linear:prices}
\end{align}
\item Clearing condition for each commodity:
\begin{align}
&\hphantom{\forall t\in T:\qquad} \sum_{i\in I} w^*_i = 0\label{eq:cc:money}\\
&\forall t\in T:\qquad \sum_{i\in I} q^*_{i,t} = 0\label{eq:cc:goods}
\end{align}
\end{itemize}
If \eqref{eq:linear:prices} and \eqref{eq:cc:goods} are satisfied,
then the clearing condition of money \eqref{eq:cc:money} is also satisfied.
\end{prop}
\begin{proof}
The utility maximization problem of each participant is given by
\begin{align}
\max\quad&w_i+\Phi_i(q_i)\label{eq:b7},\\
\st\quad&p^*_0 w_i + (p^*_T)\tp q_i\le 0\label{eq:budget},\\
        &q_i\in Q_i,\\
        &w_i\in\R\label{eq:b10},
\end{align}
where $p^*_0=1$. At first we will show, that an optimal solution to 
\eqref{eq:b7}-\eqref{eq:b10} satisfies \eqref{eq:b3}-\eqref{eq:linear:prices}.
Let $(w\s_i,q\s_i)$ be an optimal solution to \eqref{eq:b7}-\eqref{eq:b10},
then $(-(p^*_T)\tp q\s_i,\;q\s_i)$ is a feasible solution to the same problem
and we have
\begin{align*}
&w\s_i+\Phi_i(q\s_i)
  &&\mid (w\s_i,q\s_i)\text{ maximizes \eqref{eq:b7}-\eqref{eq:b10}},\\
&\ge -(p^*_T)\tp q\s_i+\Phi_i(q\s_i)
  &&\mid (w\s_i,q\s_i)\text{ satisfies \eqref{eq:budget}: }-(p^*_T)\tp q\s_i\ge w\s_i,\\
&\ge w\s_i+\Phi_i(q\s_i).
\end{align*}
As the terms in the first and last row are equal, the inequalities are
satisfied with equality. This yields that $w\s_i=-(p^*_T)\tp q\s_i$,
thus \eqref{eq:linear:prices} is satisfied.
Let $q_i\in Q_i$ be arbitrary, then $(-(p^*_T)\tp q_i,\;q_i)$
is feasible to \eqref{eq:b7}-\eqref{eq:b10} and we have
\begin{align*}
&\Phi_i(q\s_i)-(p^*_T)\tp q\s_i
   &&\mid (w\s_i,q\s_i)\text{ satisfies \eqref{eq:linear:prices}},\\
&= \Phi_i(q\s_i)+w\s_i
   &&\mid (w\s_i,q\s_i)\text{ maximizes \eqref{eq:b7}-\eqref{eq:b10}},\\
&\ge \Phi_i(q_i)-(p^*_T)\tp q_i
\end{align*}
This yields that $q\s_i$ is an optimal solution to 
$\max\left\{\Phi_i(q_i)-(p^*_T)\tp q_i \mid q_i\in Q_i\right\}$,
thus it satisfies \eqref{eq:b3}.

We will now show, that a solution that satisfies \eqref{eq:b3}-\eqref{eq:linear:prices}
is an optimal solution to \eqref{eq:b7}-\eqref{eq:b10}.
Let $q\s_i$ be an optimal solution to 
$\max\left\{\Phi_i(q_i)-(p^*_T)\tp q_i \mid q_i\in Q_i\right\}$
and let $w\s_i=-(p^*_T)\tp q\s_i$. Let $(w_i,q_i)$ be a feasible solution
to \eqref{eq:b7}-\eqref{eq:b10}, then we have
\begin{align*}
&\Phi_i(q\s_i)+w\s_i\\
&=\Phi_i(q\s_i)-(p^*_T)\tp q\s_i
  &&\mid q\s_i\text{ satisfies \eqref{eq:b3}},\\
&\ge\Phi_i(q_i)-(p^*_T)\tp q_i
  &&\mid (w_i,q_i)\text{ satisfies \eqref{eq:budget}:} -(p^*_T)\tp q_i\ge w_i,\\
&\ge\Phi_i(q_i)+w_i.
\end{align*}
This yields, that $(w\s_i,q\s_i)$ is an optimal solution to \eqref{eq:b7}-\eqref{eq:b10}.

Let \eqref{eq:linear:prices} and \eqref{eq:cc:goods} be satisfied, then
\[
\sum_{i\in I} w^*_i = -\sum_{i\in I}(p^*_T)\tp q^*_i
 = -\sum_{i\in I}\sum_{t\in T} {p^*_t} q^*_{i,t}
 = -\sum_{t\in T} {p^*_t} \sum_{i\in I} q^*_{i,t} = 0.
\]
\end{proof}

\subsection{Decision Sets and Quantity Functions}

In many cases it is convenient to express the valuation function $\Phi_i$
(cost $\le0$ / benefit $\ge0$) with the help of decision variables
(control variables) instead of quantity variables.
For this reason we introduce the concept of quantity functions
and valuation functions depending on decision variables.

Again, let $L$ be the set of tradeable commodities including the numéraire
and let $T:=L\setminus\{0\}$ be the set of commodities without the numéraire.

\begin{defn}
Let $i$ be a participant with a quasi-linear utility function $u_i(w,q)=w+\Phi_i(q)$
and let his quantity set be given by $X_i=\R\times Q_i$ with $Q_i\subseteq\R^T$.
If there is a set $D_i\subseteq\R^{n_i}$, a function $f_i:D_i\rightarrow\R^T$
and a function $v_i:D_i\rightarrow\R$ such that
\begin{align*}
Q_i&=f_i(D_i)\quad\und\\
\forall q\in Q_i:\quad 
\Phi_i(q)&=\max\{ v_i(\delta) \mid f_i(\delta)=q,\; \delta\in D_i \}.
\end{align*}
then $D_i$ is called \emph{decision set}, $f_i$ is called \emph{quantity function},
and $v_i$ is called \emph{decision valuation} of participant $i$.
\end{defn}

In \cite{blumrosenNisan2007} the evaluation of the quantity-parameterized 
optimization problem $\Phi_i(q)=\max\{ v_i(\delta) \mid f_i(\delta)=q,\; \delta\in D_i \}$
is called a \emph{value query}: ``The auctioneer presents a bundle $q$,
the bidder reports his value $\Phi_i(q)$ [in numéraire units] for this
bundle''. In a value query the participant has to choose his decision
variables such that he can buy / sell exactly the specified quantities
($f_i(\delta)=q$). Subject to this constraint he will
choose his decision variables such that they maximize an individual 
valuation function $v_i$ depending on his decision variables.
Then he will return the optimal value to the auctioneer.
For example a producer will try to find the cheapest production schedule
to produce the quantities $q$. Then he will return the production costs ($<0$)
to the auctioneer.

\begin{prop}\label{prop:strict:competitive:equilibrium:D}
Let the utility functions of all participants $i\in I$ be quasi-linear,
i.e., $u_i(w,q)=w+\Phi_i(q)$.
Let the quantity sets of all participants $i\in I$ be given by $X_i=\R\times Q_i$
with $Q_i\subseteq\R^T$. Let $D_i\subseteq\R^{n_i}$ be the decision set,
$f_i$ the quantity set, and $v_i$ the decision valuation of participant $i\in I$, that is,
\begin{align*}
Q_i&=f_i(D_i)\quad\und\\
\forall q\in Q_i:\quad \Phi_i(q)&=\max\{ v_i(\delta) \mid f_i(\delta)=q,\; \delta\in D_i \}.
\end{align*}
Let the price of money (the numéraire commodity) amount to $p^*_0=1$ currency unit.

Then the allocation $\left((w_i^*,q_i^*)\right)_{i\in I}$ and
the price vector $(1,p^*_T)\in\R^L$ constitute a \emph{competitive equilibrium}
with respect to a \emph{strict linear pricing schedule}
if and only if the following conditions hold.
\begin{itemize}
\item Utility maximization for each participant:
\begin{align}
&\forall i\in I:\qquad \valdelta_i\in\arg\max\left\{v_i(\delta_i)-(p^*_T)\tp f_i(\delta_i)\mid \delta_i\in D_i\right\}
\label{eq:17}\\
&\forall i\in I:\qquad q^*_i=f_i(\valdelta_i)
\label{eq:18}\\
&\forall i\in I:\qquad w^*_i=-(p^*_T)\tp q^*_i\label{eq:D:linear:prices}
\end{align}
\item Clearing condition for each commodity:
\begin{align}
&\hphantom{\forall t\in T:\qquad} \sum_{i\in I} w^*_i = 0\label{eq:D:cc:money}\\
&\forall t\in T:\qquad \sum_{i\in I} q^*_{i,t} = 0\label{eq:D:cc:goods}
\end{align}
\end{itemize}
Remember that the clearing condition of money \eqref{eq:D:cc:money}
is redundant and can be omitted.
\end{prop}

Before we will proof this proposition we present a model transformation
that holds in a general setting. Similar to extended formulations
the transformation allows us to express the feasible region 
(quantity set) with the help of additional (or less) variables.
The transformation also expresses the objective with
the help of these additional variables. This is extremely useful
as it allows us to model certain non-linear objectives with
the help of mixed-integer-formulations.

\newcommand\capD{}
\begin{thm}\label{thm:decision:space}
Let $Q\subseteq\R^m$ and $D\subseteq\R^n$ be arbitrary sets and let $\Phi:Q\rightarrow\R$,
$v:D\rightarrow\R$, and $f:D\rightarrow Q$ be arbitrary functions with
\begin{align*}
f(D)&=Q,\\
\forall q\in Q:\quad \Phi(q)&=\max\{ v(\delta) \mid \delta\in f\inv(q)\capD \},
\end{align*}
where $f\inv(q):=\{\delta\in D\mid f(\delta)=q\}$ is the \emph{fiber} of $f$ at $q$.
Then we have
\begin{align*}
\max_{q\in Q}\Phi(q)&=\max_{\delta\in D}v(\delta),\\
\arg\max_{q\in Q}\Phi(q)&=f\left( \arg\max_{\delta\in D}v(\delta) \right).
\end{align*}
The next figure visualizes the relationship between the different sets and functions.
\[
\newcommand\tmpB{\argmax v(D)}
\newcommand\tmpA{\argmax \Phi(Q)}
\newcommand\tmpC{\max\Phi(Q)=\max v(D)}
\xymatrix{
Q \ar[dd]|{\argmax\Phi(\cdot)}\ar[dr]|{\max{\Phi(\cdot)}}          &    &
D \ar[dd]|{\argmax v(\cdot)}  \ar[dl]|{\max v(\cdot)}     \ar[ll]|f\\
                 & \tmpC &                         \\
\tmpA\ar[ur]|\Phi&       & \tmpB\ar[ll]|f\ar[lu]|v
 }
\]
\end{thm}
\begin{proof}
In the following the sets of optimal solutions will be denoted by
$Q\s:=\arg\max_{q\in Q}\Phi(q)$ and $D\s:=\arg\max_{\delta\in D}v(\delta)$.\\
\framebox{$Q\s\subseteq f(D\s)$}
Let $q\s\in Q\s$, that is, $\Phi(q\s)\ge\Phi(q)$ for all $q\in Q$.
The function $\Phi$ is finite at all points $q\in Q$. In particular $\Phi(q\s)$
is finite and there exists $\valdelta\in f\inv(q\s)\capD$ with $\Phi(q\s)=v(\valdelta)$.
We want to show that $\valdelta$ is in $D\s$. Let $\delta\in D$ be arbitrary, then
$\delta\in f\inv(f(\delta))\capD$ holds.
\[
v(\delta) \explain{\le}{\delta\in f\inv(f(\delta))\capD}
\vect{\max&v(\delta') \hfill\\\hfill \st&\delta'\in f\inv(f(\delta))\capD}
=\Phi(f(\delta)) \explain{\le}{f(\delta)\in Q}
\Phi(q\s)=v(\valdelta).
\]
Now we have $\valdelta\in D\s$.\\
\framebox{$Q\s\supseteq f(D\s)$} Let $\valdelta\in D\s$, that is, $v(\valdelta)\ge v(\delta)$ 
for all $\delta\in D$. We know that $f(\valdelta)\in Q$ and we want to show that it is in $Q\s$.
Let $q\in Q$ be arbitrary.
\begin{align*}
\Phi( f(\valdelta) )
&=\vect{\max& v(\delta) \hfill\\\hfill \st& \delta\in f\inv( f(\valdelta) )\capD}
\explain{=}{ \valdelta\in f\inv( f(\valdelta) )\capD }
\vect{\max& v(\delta) \hfill\\\hfill \st& \delta\in D}
\ge\vect{\max& v(\delta) \hfill\\\hfill \st& \delta\in f\inv(q)\capD}
\\
&=\Phi(q).
\end{align*}
Now we have $f(\valdelta)\in Q\s$.\\
\framebox{$\max_{q\in Q}\Phi(q)=\max_{\delta\in D}v(\delta)$}
Let $\valdelta\in D\s$ and $q\s\in Q\s$, then $f(\valdelta)\in Q\s$.
\[
\Phi(q\s)
\explain{=}{f(\valdelta)\in Q\s}
\Phi(f(\valdelta))
=\vect{\max& v(\delta) \hfill\\\hfill \st& \delta\in f\inv( f(\valdelta) )\capD}
\explain{=}{ \valdelta\in f\inv( f(\valdelta) )\capD }
\vect{\max& v(\delta) \hfill\\\hfill \st& \delta\in D}
=v(\valdelta).
\]
\end{proof}

\begin{cornr}\label{cor:13}
Let $Q\subseteq\R^m$ and $D\subseteq\R^n$ be arbitrary sets and let $\Phi:Q\rightarrow\R$,
$v:D\rightarrow\R$, $f:D\rightarrow Q$, and $g:Q\rightarrow\R$ be arbitrary functions with
\begin{align*}
f(D)&=Q,\\
\forall q\in Q:\quad \Phi(q)&=\max\{ v(\delta) \mid \delta\in f\inv(q)\capD \},
\end{align*}
where $f\inv(q):=\{\delta\in D\mid f(\delta)=q\}$ is the \emph{fiber} of $f$ at $q$.
Then we have
\begin{align}
\max_{q\in Q}\Phi(q)-g(q)&=\max_{\delta\in D}v(\delta)-g(f(\delta)),\\
\arg\max_{q\in Q}\Phi(q)-g(q)&=f\left( \arg\max_{\delta\in D}v(\delta)-g(f(\delta)) \right).
\label{eq:23}
\end{align}
\end{cornr}
\begin{proof}
Let $\widetilde\Phi(q):=\Phi(q)-g(q)$ and $\widetilde v(\delta):=v(\delta)-g(f(\delta))$.
For all $q\in Q$ we have
\begin{align*}
\widetilde\Phi(q)&=\Phi(q)-g(q)\\
	&=\vect{\max &v(\delta) \hfill\\\hfill \st & \delta\in f\inv(q)\capD} - g(q)
&&|\;g(q) \text{ is not depending on }\delta,\\
	&=\vect{\max &v(\delta)-g(q) \hfill\\\hfill \st & \delta\in f\inv(q)\capD}
&&|\;q=f(\delta)\text{ holds for all }\delta\in f\inv(q)\capD,\\
	&=\vect{\max &\widetilde v(\delta) \hfill\\\hfill \st & \delta\in f\inv(q)\capD}.
\end{align*}
We can now apply Theorem \ref{thm:decision:space} to the functions
$\widetilde\Phi$, $\widetilde v$, and $f$.
\end{proof}

\begin{proof}[Proof of Prop. \ref{prop:strict:competitive:equilibrium:D}]
The notation of Proposition \ref{prop:strict:competitive:equilibrium:D} is 
consistent with the notation of Corollary \ref{cor:13}. We only need to
define the function $g(q):=(p^*_T)\tp q$. The left hand side of equation 
\eqref{eq:23} is equivalent to equation \eqref{eq:b3} whereas the right
hand side of \eqref{eq:23} is equivalent to \eqref{eq:17} and \eqref{eq:18}.
\end{proof}

We are now ready to introduce a succinct definition of a competitive
equilibrium that only involves decision sets and quantity functions.
The definition is consistent with the previous definitions and propositions.
In particular it is consistent with Proposition \ref{prop:strict:competitive:equilibrium:D}.

\begin{defn}
Let $T$ be the set of tradable commodities without the numéraire.
Let $D_i\subseteq\R^{n_i}$ be the decision set, $f_i:D_i\rightarrow\R^T$
the quantity function, and $v_i:D_i\rightarrow\R$ the decision valuation of
participant $i\in I$.
Remember that participant $i$ has a quasi linear utility function.

Then the tuple $\left(\valdelta\right)_{i\in I}$ and
the price vector $p^*\in\R^T$ constitute a \emph{competitive equilibrium}
with respect to a \emph{strict linear pricing schedule}
if and only if the following three conditions hold.
\begin{itemize}
\item Utility maximization for each participant:
\begin{align}
&\forall i\in I:\qquad \valdelta_i\in\arg\max\left\{v_i(\delta_i)-(p^*)\tp f_i(\delta_i)\mid \delta_i\in D_i\right\}
\label{eq:walras:individual}
\end{align}
\item Strict Linear Pricing Schedule:
\begin{align}
\forall i\in I:\qquad w^*_i:=-(p^*)\tp f_i(\valdelta_i)\label{eq:25}
\end{align}
\item Clearing condition for each commodity in $T$:
\begin{align}
&\sum_{i\in I} f_i(\valdelta_i) = 0\label{eq:26}
\end{align}
\end{itemize}
The term $f_i(\valdelta_i)$ is the quantity vector of participant $i$ and 
$w^*_i$ is the amount of money to be payed ($<0$) or received ($>0$).
The clearing condition of money
$\sum_{i\in I} w^*_i = 0$ is implied by \eqref{eq:25} and \eqref{eq:26}.
\end{defn}

\section{Competitive Equilibria in Convex Auctions}
\label{sec competitive equilibria}

In this section we study a fundamental property of convex auctions.
A \emph{convex auction} is an auction where all participants submit
convex bid parameters. If the auction is convex then we are able to find
a competitive equilibrium by just solving a convex optimization problem.

In the previous sections the individual optimization problems of the
participants where expressed with the help of the bid parameters
$(D_i,v_i,f_i)$: the decision set, the decision valuation, and the quantity
function. This time we will assume that all participants express their bids
by submitting convex bid parameters:

\begin{defn}
Let $T$ be the set of commodities without the numéraire.
The bid parameters $(D_i,v_i,f_i)$ of a participant $i$ are 
\emph{convex} if the decision valuation $v_i:D_i\rightarrow\R$
is concave and differentiable, the quantity function $f_i:D_i\rightarrow\R^T$
is affine and the decision set is given by 
\begin{align*}
D_i=\{\delta_i\in\R^{n_i}\mid g_{i,k}(\delta_i)\le 0\fuer k=1,\dots,m_i\},
\end{align*}
where $g_{i,k}:\R^{n_i}\rightarrow\R$ are convex differentiable functions.
We will use the tuple $(g_i,v_i,f_i)$ to denote convex bid parameters.
A bid is \emph{convex} if its parameters are convex.
\end{defn}

\newcommand\welfarethm{fundamental welfare theorem}
\begin{thm}[First Fundamental Welfare Theorem]\label{thm fundamental welfare}
Let $T$ be the set of commodities without the\linebreak[4] numéraire.
Let the bids of all participants $i\in I$ be expressed by convex bid parameters
$(g_i,v_i,f_i)$ and let the weak Slater assumption hold for the constraints
\eqref{eq:fwt:2}-\eqref{eq:fwt:3}.

Then the tuple $(\valdelta_i)_{i\in I}$ and
the price vector $\p^*\in\R^T$ constitute a \emph{competitive equilibrium}
with respect to a \emph{strict linear pricing schedule}
if and only if $\valdelta$ is an optimal solution to
\begin{align}
\max\quad&\sum_{i\in I}v_i(\delta_i)\label{eq:fwt:1},\\
\st\quad &\sum_{i\in I}f_i(\delta_i)=0&&[\pi],\label{eq:fwt:2}\\
         &\delta_i\in D_i&&[\lambda_i]&&\forall i\in I\label{eq:fwt:3},
\end{align}
\newcommand{\welfareOP}{\eqref{eq:fwt:1}-\eqref{eq:fwt:3}}%
and $(\valp,\vallambda)$ is an optimal dual solution. \p\ is the dual variable
to the clearing condition and $\lambda_{i,k}$ is the dual variable to the decision
set constraints $g_{i,k}(\delta_i)\le 0$. An optimal solution to this problem
is called a \emph{welfare maximizing solution}.
\end{thm}
\begin{proof}
As the weak Slater assumption holds for \eqref{eq:fwt:2}-\eqref{eq:fwt:3} it also
holds for the individual optimization problems \eqref{eq:walras:individual}. This
yields that in both cases the KKT conditions are necessary and sufficient to describe
optimal solutions. The proof is basically a straightforward reformulation of the
KKT conditions.

At first we write out the optimization problem \eqref{eq:fwt:1}-\eqref{eq:fwt:3} in detail:
\begin{align*}
\max\quad&\sum_{i\in I}v_i(\delta_i),\\
\st\quad &\sum_{i\in I}f_{i,t}(\delta_i)=0&&[\pi_t]&&\forall \hours,\\
         &g_{i,k}(\delta_i)\le0&&[\lambda_{i,k}]&&\forall i\in I, k=1,\dots,m_i,\\
         &\delta_{i,j}\in\R&&&&\forall i\in I,j=1,\dots,n_i.
\end{align*}
A solution $\valdelta$ is optimal to this problem if and only if there exist
variables $\valp$ and $\vallambda$ such that
\begin{align}
\sum_{i\in I}f_{i,t}(\valdelta_i)&=0&&\forall\hours,\\
g_{i,k}(\valdelta_i)&\le0&&\forall i\in I, k= 1,\dots,m_i,
    \label{eq walras primal feas}\\
\sum_\hours\valp_t\frac{\partial f_{i,t}}{\partial\delta_{i,j}}(\valdelta_i)
+\sum_{k=1}^{m_i}\vallambda_{i,k}\frac{\partial g_{i,k}}{\partial\delta_{i,j}}(\valdelta_i)
&=\frac{\partial v_i}{\partial\delta_{i,j}}(\valdelta_i)&&\forall i\in I,j=1,\dots,n_i,
    \label{eq walras dual feas 1}\\
\vallambda_{i,k}&\ge0&&\forall i\in I, k= 1,\dots,m_i,
    \label{eq walras dual feas 2}\\
g_{i,k}(\valdelta_i)\vallambda_{i,k}&=0&&\forall i\in I, k=1,\dots,m_i
    \label{eq walras complementarity}.
\end{align}
Equations \eqref{eq walras primal feas}, \eqref{eq walras dual feas 1}-%
\eqref{eq walras dual feas 2}, and \eqref{eq walras complementarity} correspond
to the primal feasibility, the dual feasibility, and the complementarity condition
of the individual optimization problems \eqref{eq:walras:individual}.
This yields the following reformulation:
\begin{align*}
&\sum_{i\in I}f_{i,t}(\valdelta_i)=0&&\forall\hours,\\
&\left(\matr{
\valdelta_i\in\arg\max&\quad v_i(\delta_i) - \sum_\hours \valp_t f_{i,t}(\delta_i)\\
\hfill\st&\quad g_{i,k}(\delta_i)\le 0\qquad\hfill&\forall k=1,\dots,m_i
}\right)&&\forall i\in I.
\end{align*}
Note that the price $\valp$ is an exogenously given parameter in the individual 
optimization problem. In other words:
\begin{align*}
&\sum_{i\in I}f_{i,t}(\valdelta_i)=0&&\forall\hours,\\
&\valdelta_i\in\arg\max\{v_i(\delta_i) - {\valp}\tp f_i(\delta_i)
\mid \delta_i\in D_i\}&&\forall i\in I.
\end{align*}
These two equations in conjunction with the linear pricing schedule 
$w^*_i:=-{\valp}\tp f_i(\valdelta_i)$ for all participants $i\in I$ 
yield that $\valdelta$ and the price vector $\valp$ constitute a 
competitive equilibrium with respect to a strict linear pricing schedule.
\end{proof}

\section{Competitive Equilibrium Relaxations in Non-Convex Auctions}
\label{sec competitive relaxations}

The previous section introduced the \welfarethm\ 
that allows us to compute a competitive equilibrium in a convex auction 
by solving a convex optimization problem.
As soon as at least one participant submits non-convex bid
parameters, the auction becomes \emph{non-convex}. In such cases
a competitive equilibrium with respect to a strict linear pricing schedule
might not exist, as the following example shows.

\begin{example}\label{ex there is no walrasian equilibrium}
There is one buyer and one seller. The buyer wants to buy at most one 
quantity unit and his benefit per unit amounts to four currency units 
per quantity unit. He submits the convex bid parameters
$(D_0=[0,1],\ v_0(\delta_0)=4\delta_0,\ f_0(\delta_0)=\delta_0)$.
The seller wants to sell exactly two units or no unit at all and his
cost per quantity unit amounts to three currency units per quantity unit.
He submits the non-convex bid parameters
$(D_1=\{-2,0\},\ v_1(\delta_1)=3\delta_1,\ f_1(\delta_1)=\delta_1)$.
The only feasible solution is $\delta_0=\delta_1=0$. 
We will see that there is no price $\p_0$ such that the tuple of decision
vectors $(\delta_0,\delta_1)$ and the price $\p_0$ constitute a competitive
equilibrium: If the price is strictly lower than four currency units,
then the optimal strategy of the buyer is to buy exactly one unit
which is not possible. If the price is strictly greater than three currency
units, then the optimal strategy of the seller is to sell exactly two units
which is not possible. Even though the solution is not a competitive
equilibrium, we can observe that at least no participant incurs a loss.
\end{example}

Regardless of whether a competitive equilibrium exists or not,
we can determine a welfare maximizing solution by solving
the model in Theorem \ref{thm fundamental welfare}. However the next example
shows that a welfare maximizing solution does not necessarily admit a strict
linear pricing schedule where no participant incurs a loss.

\begin{example}\label{ex surplus maximizing solution}
There are two buyers and one seller.
Buyer $1$ wants to buy at most one unit if the price is lower or equal to four.
Buyer $2$ wants to buy at most two units if the price is lower or equal to six.
Seller $1$ wants to sell exactly three or no units if the price is 
greater or equal to five.
\begin{align*}
\begin{aligned}
\text{buyer }1:\quad&(D_1=[0,1],v_1(\delta_1)=4\delta_1,f_1(\delta_1)=\delta_1),\\
\text{buyer }2:\quad&(D_2=[0,2],v_2(\delta_2)=6\delta_2,f_2(\delta_2)=\delta_2),\\
\text{seller }1:\quad&(D_3=\{-3,0\},v_3(\delta_3)=5\delta_3,f_3(\delta_3)=\delta_3).\\
\end{aligned}
&&
\begin{aligned}
\max\quad&4\delta_1+6\delta_2+5\delta_3,\\
\st\quad&\delta_1+\delta_2+\delta_3=0,\\
&\delta_1\in[0,1],\\
&\delta_2\in[0,2],\\
&\delta_3\in\{-3,0\}.
\end{aligned}
\end{align*}
The welfare maximizing solution $(\delta_1^*,\delta_2^*,\delta_3^*)$ is $(1,2,-3)$.
If the price $\p_0$ is strictly greater than four, then buyer 1 incurs a loss, as
he is only willing to pay at most four CU/QU (currency units per quantity unit).
If the price is strictly less than five, then seller 1 incurs a loss, as he
wants to receive at least five CU/QU.
\end{example}

As shown by the previous examples, a competitive equilibrium might not exist
and a welfare maximizing solution does not necessarily admit a strict linear
pricing schedule, where no participant incurs a loss.
This implies that in general there is no $(\delta,\p)$ such that
\begin{align*}
&\sum_{i\in I}f_{i,t}(\delta_i)=0&&\forall\hours,\\
&\delta_i\in\arg\max\{v_i(\delta_i') - {\p}\tp f_i(\delta_i')
\mid \delta_i'\in D_i\}&&\forall i\in I.
\end{align*}
In practice a lot of exchanges facilitate the participants to submit non-convex bid parameters.
A popular non-convex bid in stock exchanges is the so called \emph{fill-or-kill limit order}.
The sell order in Example \ref{ex there is no walrasian equilibrium} is such a non-convex
fill-or-kill limit order, whereas the buy order is a convex \emph{limit order}. In other
words the example is a small auction at a stock exchange and it shows that
in general stock exchanges cannot publish execution schedules and strict linear prices 
that constitute a competitive equilibrium. In fact most of the exchanges are publishing
strict linear prices. This implies that they relax some of the optimality conditions of
the individual optimization problems to make the previous optimization problem feasible.
It is clear that there are various possibilities for relaxing optimality conditions and
for choosing a solution within these newly introduced degrees of freedom.

We will present two closely related approaches. The first one shows that it is sufficient
to relax some of the optimality conditions: the execution state of a non-convex 
bid either maximizes the individual surplus or the bid is rejected. The second approach 
is a pragmatic one: it ensures that the surplus of non-convex bids is non-negative.

\subsection{Model A: Individual Surplus Maximization or Rejection}
\label{sec maximization or rejection}

Again let $T$ be the set of commodities without the numéraire.
Let $I=I\convex\dot\cup I\nonconvex$ be the set of all bids, at which $I\convex$
is the set of convex bids and $I\nonconvex$ is the set of non-convex bids.
The bid parameters of bid $i\in I\convex\cup I\nonconvex$ are denoted by
$(D_i,v_i,f_i)$. Consider the model
\begin{align}
\max\quad
&\sum_{i\in I}v_i(\delta_i),
\label{eq nonconvex rejection relaxation 0}\tag{MainMPEC}\\
\st\quad
\label{eq convex competitive equilibrium 1}
&\sum_{i\in I}f_{i,t}(\delta_i)=0
    &&\forall\hours,\\
&\delta_i\in\arg\max\{v_i(\delta_i') - {\p}\tp f_i(\delta_i')
\label{eq convex competitive equilibrium 2}
\mid \delta_i'\in D_i\}
    &&\forall i\in I\convex,\\
\label{eq nonconvex rejection relaxation 3}
&\delta_i\in\arg\max\{v_i(\delta_i') - {\p}\tp f_i(\delta_i')
\mid \delta_i'\in D_i\}\cup\ker f_i
    &&\forall i\in I\nonconvex.
\end{align}
\newcommand\rejectionMPEC{\eqref{eq nonconvex rejection relaxation 0}}%
The Greek letters $\delta$ and $\p$ denote the variables of the model.
If there are only convex bids, then $I\nonconvex=\emptyset$ and equations
\eqref{eq convex competitive equilibrium 1}-\eqref{eq convex competitive equilibrium 2}
sufficiently describe a welfare maximizing solution.
In equation \eqref{eq nonconvex rejection relaxation 3} we extended the feasible
region by adding the kernel of the quantity functions.
This relaxation enables us to reject non-convex bids independent of the prices $\p$:
If the decision variables of the non-convex bids are
in the kernel of the respective quantity function (i.e., $\delta_i\in\ker f_i$
for $i\in I\nonconvex$), then in the clearing condition 
\eqref{eq convex competitive equilibrium 1} the term $f_{i,t}(\delta_i)$ vanishes for all 
non-convex bids $i\in I\nonconvex$. This yields that the model is feasible if there is a
solution that only involves the convex bids. If for example for each convex bid there
is a feasible decision vector such that the quantity function maps to zero
(i.e., $\ker f_i\cap D_i\neq\emptyset$ for $i\in I\convex$),
then the model is feasible.

In example \ref{ex there is no walrasian equilibrium} the model would select
a solution that satisfies $\delta_0=\delta_1=0$ and $\p_0\ge4$: Let for example
$\p_0=4$ then $\arg\max\{v_0(\delta_0') - \p\tp f_0(\delta_0')\mid \delta_0'\in D_0\}
=\arg\max\{4\delta_0' - 4 \delta_0'\mid \delta_0'\in [0,1]\}=[0,1]\ni\delta_0$
and $\ker f_1=\{0\}\ni\delta_1$. Note that this solution is not a competitive 
equilibrium because at price four the optimal strategy of the seller is
to sell exactly two units instead of selling nothing.

In example \ref{ex surplus maximizing solution} the model cut off the
welfare maximizing solution and return a solution that satisfies
$\delta_1=\delta_2=\delta_3=0$ and $\p_0\ge 6$.

In general we cannot find a competitive equilibrium, but we can find 
strict linear prices where the decision variables of all convex bids
are profit maximizing:
\begin{cornr}\label{cor prices for convex bids}
Let $T$ be the set of commodities without the numéraire and let $(g_i,v_i,f_i)$
be the bid parameters of convex bids $i\in I\convex$ and $(D_i,v_i,f_i)$
the bid parameters of non-convex bids $i\in I\nonconvex$.

Let $\valdelta$ be an optimal solution to
\begin{align}
\max\quad
&\sum_{i\in I}v_i(\delta_i),
    \label{cor prices for convex bids eq 0}
    \tag{MaxWelfare}\\
\st\quad
&\sum_{i\in I}f_{i,t}(\delta_i)=0
    &&\forall\hours,\\
&g_i(\delta_i)\le0
    &&\forall i\in I\convex,\\
&\delta_i \in D_i
    &&\forall i\in I\nonconvex.
\end{align}
Let the weak Slater assumption hold for 
$\sum_{i\in I\convex}f_{i}(\delta_i)=-\sum_{i\in I\nonconvex}f_{i}(\valdelta_i)$
and $\forall i\in I\convex:g_i(\delta_i)\le 0$.
The assumption holds if for example all convex decision sets are polyhedrons.
Then there exists $\p\in\R^T$ with
\begin{align*}
\valdelta_i\in\arg\max\{v_i(\delta_i) - {\p}\tp f_i(\delta_i)
\mid g_i(\delta_i)\le 0\}
    &&\forall i\in I\convex. 
\end{align*}
In other words: for all $i\in I\convex$ there exists $\p$ and $\mu_i$ with
\begin{align*}
\mu_i\tp g_i(\valdelta_i) = 0
&&\wedge&&
\mu_i\tp \DD g_i(\valdelta_i)=\DD v_i(\valdelta_i)-\p\tp \DD f_i(\valdelta_i)
&&\wedge&&
\mu_i\ge 0.
\end{align*}
\end{cornr}
\newcommand{\maxWelfare}{\eqref{cor prices for convex bids eq 0}}%
\begin{proof}
Let $\valdelta$ be an optimal solution. We will replace all non-convex
bids by ``constant'' convex bids, such that the fundamental welfare theorem 
yields the desired result.
For all $i\in I\nonconvex$ let $D_i^*=\{\valdelta_i\}$, 
$f_i^*(\delta)=f_i(\valdelta)$, and $v_i^*(\delta)=v_i(\valdelta)$.
Each bid $(D_i^*,v_i^*,f_i^*)$ is convex, as the decision set is a singleton
and the decision valuation and the quantity function are constant functions.
If we replace all non-convex bids by these constant bids, then $\valdelta$
is still an optimal solution to the modified problem. As the modified
problem is a convex auction, we can apply the fundamental welfare theorem.
\end{proof}

In the following we present an algorithm that exploits this property 
and ensures the optimality conditions without modeling them explicitly.
Similar to the generalized Benders decomposition (\cite{geoffrion1972}),
the algorithm decomposes the problem into a master problem and a subproblem.
Readers who are not interested in algorithmic details can safely skip
the rest of this subsection.

In our applications the individual optimization problems of non-convex 
bids are bounded mixed integer programs, that is, the decision valuations
and the quantity functions are affine linear functions and the decision sets
are feasible regions of bounded mixed integer programs. In our case
these mixed integer programs are very small such that we can specify
the polyhedral convex hull of each non-convex decision set.
This property is crucial for the algorithm and points out an important
computational limitation.

For the sake of exposition we will focus only on mixed integer auctions.
In a \emph{mixed integer auction} all bids are either 
convex or mixed integer bids. Mixed integer bids are introduced in
\newcommand\MIBparams{$((A_i,a_i,z_i),c_i,Q_i)$}
\begin{defn}
Let $T$ be the set of commodities without the numéraire.
A bid $i$ with parameters $(D_i,v_i,f_i)$ is a \emph{mixed integer bid}
if $D_i$ is bounded and there exist parameters \MIBparams\ with
\begin{align*}
   v_i(\delta_i)=c_i\tp\delta_i,
 &&f_i(\delta_i)=Q_i\delta_i,
 &&D_i=\{\delta_i\in\R^{n_i-z_i}\times\Z^{z_i}\mid A_i\delta_i\le a_i\},
 &&\und z_i\ge1.
\end{align*}
The price parameterized individual optimization problem is given by
\begin{align}
  \max\quad& c_i\tp\delta_i-\pi\tp Q_i\delta_i,\tag{MIBid)(\p}\\
  \st\quad&A_i\delta_i\le a_i,\\
          &\delta_i\in \R^{n_i-z_i}\times\Z^{z_i}.
\end{align}
\end{defn}

\begin{thm}\label{thm max or reject}
Let $v_i(\delta_i)=c_i\tp\delta_i$, $f_i(\delta_i)=Q_i\delta_i$, $D_i\subseteq\Rn$ bounded,
and let $\conv(D_i)=\{\delta_i\in\Rn \mid A_i\delta_i\le a_i\}$.
Then
\[
 \valdelta_i\in
 \arg\max\{c_i\tp\delta_i-\pi\tp Q_i\delta_i \mid \delta_i\in D_i\}\cup \{0\}
\]
if and only if there exists $\vallambda_i$ such that
\begin{align*}
&\valdelta_i\in
 \arg\max\{c_i\tp\delta_i-\pi\tp Q_i\delta_i \mid A_i\delta_i\le \vallambda_i a_i\},\\
&\valdelta_i\in D_i\cup\{0\},\\
&\vallambda_i\in\bin.
\end{align*}
\end{thm}

\begin{lem}\label{thm max or reject lem 1}
Let $D\subseteq\Rn$ and $c\in\Rn$. Then
\[
\arg\max\{c\tp \delta \mid \delta\in D\}
=\arg\max\{c\tp \delta \mid \delta\in\conv D\} \cap D.
\]
\end{lem}
\comment{
\begin{proof}
($\subseteq$) Let $\valdelta$ maximize $c\tp\delta$ subject to $\delta\in D$.
Then $\valdelta\in D\subseteq\conv D$ holds and for all $\delta\in D$ we have
$c\tp\delta\le c\tp\valdelta$. Recall that $\conv D=\left\{
  \sum_{i=1}^n\lambda_i\delta_i\mid
  \delta_i\in D,\ n\in\N,\ \sum_{i=1}^n\lambda_i=1,\ \lambda_i\ge 0
\right\}$. Let $\delta\in\conv D$ then there is $n\in\N$ 
and $\lambda_i\ge0,\delta_i\in D$ with
$\sum_{i=1}^n\lambda_i\delta_i$ and $\sum_{i=1}^n\lambda_i=1$. This yields
$c\tp\delta=\sum_{i=1}^n\lambda_i c\tp\delta_i
         \le\sum_{i=1}^n\lambda_i c\tp\valdelta=c\tp\valdelta$.
$\valdelta$ maximizes $c\tp\delta$ subject to $\delta\in\conv D$.\\
($\supseteq$) Let $\valdelta$ maximize $c\tp\delta$ subject to $\delta\in\conv D$
and let $\valdelta\in D$. Then for all $\delta\in\conv D$ we have
$c\tp\delta\le c\tp\valdelta$. Let $\delta\in D$, then $\delta\in\conv D$ and
$c\tp\delta\le c\tp\valdelta$. $\valdelta$ maximizes $c\tp\delta$ subject to
$\delta\in D$.
\end{proof}
}%
\begin{lem}\label{thm max or reject lem 2}
Let $X=\{x\in\Rn \mid A x\le a\}$ be a bounded polyhedron (i.e., a polytope). Then
\[
 X\cup\{0\} = \{x\in\Rn \mid Ax\le\lambda a, \lambda\in\bin\}.
\]
This is also true if we parameterize $X$: Let $\mu\in\R^m$ and 
$X(\mu)=\{x\in\Rn \mid A(\mu) x\le a(\mu)\}$ be a bounded parameterized polytope, then
\[
 X(\mu)\cup\{0\} = \{x\in\Rn \mid A(\mu)x\le\lambda a(\mu), \lambda\in\bin\}.
\]
\end{lem}
\begin{proof}
Confer \cite{balas1979disjunctive}.
\end{proof}

\begin{proof}[Proof of Theorem \ref{thm max or reject}.]
For overview purposes we will omit the indices $i$. Recall that $\p$ is
exogenously given. If $D$ is empty, then the theorem is trivial. Let
$D$ be non-empty.
\begin{align*}
&\arg\max\{c\tp\delta-\p\tp Q\delta\mid\delta\in D\} \cup \{0\}
&&\mid \text{Lemma \ref{thm max or reject lem 1}}\\
&=\left(\arg\max\{c\tp\delta-\p\tp Q\delta\mid\delta\in\conv(D)\}
    \cap D\right)\cup\{0\}
&&\mid (A\cap D)\cup N=(A\cup N)\cap(D\cup N)\\
&=\left(\arg\max\{c\tp\delta-\p\tp Q\delta\mid A\delta\le a\}\cup\{0\}\right)
    \cap (D\cup\{0\})
&&\mid \text{KKT conditions}\\
&=\left(
   \left(
      \delta \mid \matr{
                    \exists\mu\ge 0\mit\hfill\\
                    A\tp\mu=c-Q\tp\p\hfill\\
                    (A\delta-a)\tp\mu=0\\
                    A\delta\le a\hfill
                  }
  \right) \cup\{\delta\mid\delta=0\}
  \right) \cap (D\cup\{0\}) = \dots
\end{align*}
The set $D$ is bounded, thus $\{\delta\mid A\delta\le a\}$ is bounded and therefore
$X(\mu)=\{\delta\mid A\delta\le a,(\mu\tp A)\delta=(\mu\tp a)\}$ is a bounded
polyhedron for all $\mu\ge0$. Let $P(\p)=\max\{c\tp\delta-\p\tp Q\delta\mid A\delta\le a\}$
and $P'(\p)=\{\mu\ge 0\mid A\tp\mu=c-Q\tp\p\}$. We know that $P(\p)$ has a finite
optimal solution for all $\p\in\R^T$ and the strong duality yields that $P'(\p)$
has a solution for all $\p\in\R^T$.
\begin{align*}
\dots&=\left(
         \{\delta\mid\exists\mu:\mu\in P'(\p),\delta\in X(\mu)\}
         \cup\{\delta\mid\exists\mu:\mu\in P'(\p),\delta\in\{0\}\}
       \right)
         \cap (D\cup\{0\})\\
&=\left\{\delta\mid\exists\mu:\mu\in P'(\p),\delta\in(X(\mu)\cup\{0\})\right\}
         \cap (D\cup\{0\})\\
&=\left\{\delta\mid\exists\mu:\mu\in P'(\p),\delta\in(X(\mu)\cup\{0\}),\delta\in(D\cup\{0\})\right\}
=\dots
\end{align*}
Lemma \ref{thm max or reject lem 2} yields that
$X(\mu)\cup\{0\}=\{\delta\mid\exists\lambda\in\bin:
A\delta\le\lambda a,(\mu\tp A)\delta=\lambda(\mu\tp a)\}$ for all $\mu$.
\[
\dots=\left\{
  \delta\mid\matr{
      \exists\lambda,\mu\mit\hfill\\
                    \mu\ge 0\hfill\\
                    A\tp\mu=c-Q\tp\p\hfill\\
                    \lambda\in\bin\hfill\\
                    (A\delta-\lambda a)\tp\mu=0\\
                    A\delta\le\lambda a\hfill\\
                    \delta\in D\cup\{0\}\hfill
    }
\right\}.
\]
\end{proof}

We will now describe the algorithm step by step. Given a mixed integer auction,
the algorithm finds an optimal solution to \rejectionMPEC. It computes the
optimal solution by solving a sequence of relaxations.
These relaxations omit the optimality conditions, such that each relaxation
is a mixed integer convex program, a MICP. The algorithm requires that
for each mixed integer bid the convex hull of the decision set is given.
For $\val\Lambda\subseteq\bin^{I\nonconvex}$ consider the 
$\val\Lambda$-parameterized optimization problem
\begin{align}
\max\quad&\sum_{i\in I} v_i(\delta_i),
  \label{eq masterMICP 0}\tag{MasterMICP)($\val\Lambda$}\\
\st\quad
&\sum_{i\in I} f_{i,t}(\delta_i)=0
  &&\forall \hours,
  \label{eq masterMICP 1}\\
&g_{i}(\delta_i)\le 0
  &&\forall i\in I\convex,
  \label{eq masterMICP 2}\\
&A_i\delta_i\le\lambda_i a_i
  &&\forall i\in I\nonconvex,
  \label{eq masterMICP 4}\\
&\delta_i\in\R^{n_i-z_i}\times\Z^{z_i}
  &&\forall i\in I\nonconvex,
  \label{eq masterMICP 5}\\
&\lambda\in\val\Lambda,
  &&\label{eq masterMICP 6}
\end{align}
\newcommand\masterMICP{\eqref{eq masterMICP 0}}%
\newcommand\masterMICPof[1]{\hyperref[eq masterMICP 0]{\textnormal{(MasterMICP)(#1)}}}%
where $T$ is the set of commodities without the numéraire, all $i\in I\convex$ are
convex bids with parameters $(g_i,v_i,f_i)$, and all $i\in I\nonconvex$ are
mixed integer bids with parameters \MIBparams. The binary variable $\lambda_i$ models
whether a non-convex bid $i$ is rejected or not (cf. Theorem \ref{thm max or reject}).

\begin{assumption}
In the rest of this chapter we always assume 
the following:

For all mixed integer bids $i\in I\nonconvex$ we have
$\conv(D_i)=\{\delta_i\mid A_i\delta_i\le a_i\}$
and for all convex bids $i\in I$ the decision set $D_i$ is bounded.
For all $\lambda\in\bin^{I\nonconvex}$ the weak slater assumption
holds for \eqref{eq masterMICP 1}-\eqref{eq masterMICP 4}.

(For example, the last assumption holds if there is 
a feasible solution $\valdelta$ to \eqref{eq masterMICP 1}-\eqref{eq masterMICP 2}
with $g_i(\valdelta_i)<0$ for all $i\in I\convex$
and $0=\delta_i\in D_i$ for all $i\in I\nonconvex$.)
\end{assumption}

In the first step, the algorithm computes an optimal solution
$(\valdelta,\vallambda)$ to \masterMICPof{$\bin^{I\nonconvex}$}.
Without loss of generality, we may assume that $\vallambda_i=1$ for all 
$i\in I\nonconvex$. Note that if we fix all $\lambda_i$ to $1$, then the
previous model is equivalent to the \maxWelfare\ model in Corollary 
\ref{cor prices for convex bids}. In other words $(\valdelta,\vallambda)$
is a welfare maximizing solution.
Observe that whenever $(\valdelta,\vallambda)$ is an optimal solution to 
\masterMICP, we can apply Corollary \ref{cor prices for convex bids}.
This becomes clear if we use parameterized decision sets $D_i'(\vallambda)$
in the \maxWelfare\ model: for all non-convex bids $i\in I\nonconvex$ use
$D_i'(\vallambda):=\{\delta_i\mid \delta_i\text{ satisfies \eqref{eq masterMICP 4}
and \eqref{eq masterMICP 5}}\}$ instead of $D_i$.

In the next step, a linear program checks whether there exists a
strict linear pricing schedule $\p$, such that $(\valdelta,\p)$
is feasible for \rejectionMPEC. In other words, the program checks
whether there exists a strict linear pricing 
schedule such that all convex bids are profit maximizing and all
non-convex bids are either profit maximizing or rejected.
Theorem \ref{thm max or reject} provides, that it is sufficient to
check whether there is a pricing schedule $\pi$ such that
\begin{align}
\valdelta_i&\in\arg\max\{v_i(\delta_i)-\p\tp f_i(\delta_i)\mid g_i(\delta_i)\le0\}
&&\forall i\in I\convex\und
\label{eq convex max}\\
\valdelta_i&\in\arg\max\{c_i\tp\delta_i-\p\tp Q_i\delta_i\mid A_i\delta_i\le \vallambda_i a_i\}
&&\forall i\in I\nonconvex.
\label{eq nonconvex max}
\end{align}
According to Corollary \ref{cor prices for convex bids} there is
a strict linear pricing schedule such that all convex bids are 
profit maximizing, i.e., such that the first equation holds.
Recall that the welfare maximizing solution does not necessarily possess
a pricing schedule that satisfies both equations
(cf. Example \ref{ex surplus maximizing solution}).

We use the KKT conditions to reformulate equations \eqref{eq convex max}
and \eqref{eq nonconvex max}. The objective of the following program is 
zero if and only if there exists a $\p$ such that the two equations hold.
This will be discussed in the following two paragraphs.
\begin{align}
\min\quad&-\sum_{i\in I\nonconvex}\mu_i\tp(A_i\valdelta_i-\vallambda_i a_i),
\label{price lp eq 0}\tag{PriceLP)($\valdelta,\vallambda$}\\
\st\quad
&\mu_i\tp g_i(\valdelta_i) = 0
    &&\forall i\in I\convex,
    \label{price lp eq 1}\\
&\mu_i\tp \DD g_i(\valdelta_i)=\DD v_i(\valdelta_i)-\p\tp \DD f_i(\valdelta_i)
    &&\forall i\in I\convex,
    \label{price lp eq 2}\\
&\mu_i\ge 0
    &&\forall i\in I\convex,
    \label{price lp eq 3}\\
&\mu_i\tp A_i=c_i\tp-\p\tp Q_i
    &&\forall i\in I\nonconvex,
    \label{price lp eq 4}\\
&\mu_i\ge 0
    &&\forall i\in I\nonconvex.
    \label{price lp eq 5}
\end{align}
\newcommand{\priceLP}{\eqref{price lp eq 0}}%
\newcommand{\priceLPof}[1]{\hyperref[price lp eq 0]{\textnormal{(PriceLP)(#1)}}}%
In this model the terms $\valdelta$ and $\vallambda$ are exogenously given 
parameters, whereas the terms $\p$ and $\mu$ are the variables. 
Observe that the model is a linear program. The term $\p$ represents 
a strict linear price vector and $\mu_i$ corresponds to the dual variables
of the price parameterized convex optimization problems
\eqref{eq convex max} and \eqref{eq nonconvex max}.

The following paragraph shows that \priceLP\ is feasible and the
objective equals zero if and only if there is a $\p$ with 
\eqref{eq convex max} and \eqref{eq nonconvex max}. Recall that
$(\valdelta,\vallambda)$ is an optimal solution to \masterMICP.
This allows us to apply Corollary \ref{cor prices for convex bids}, which yields
that there exists $\mu$ and $\p$ such that the constraints 
\eqref{price lp eq 1}-\eqref{price lp eq 3} are satisfied. Note that these three
constraints are the complementarity condition and the dual feasibility of
\eqref{eq convex max}. It remains to be checked that for a given $\p$ the constraints
\eqref{price lp eq 4}-\eqref{price lp eq 5} are feasible. Note that these two 
constraints correspond to the dual feasibility of the price
parameterized linear programs in \eqref{eq nonconvex max}. The definition of
mixed integer bids ensures that the feasible regions of these linear programs
are bounded, thus they have a finite optimal solution for all $\pi$.
Therefore, the dual problems of \eqref{eq nonconvex max} are feasible for all $\p$.
Since the last two constraints correspond to the dual feasibility of 
\eqref{eq nonconvex max}, for all $\p$ there exists $\mu$ such that they are satisfied.
Note that the objective of \priceLP\ corresponds to the complementarity conditions
of \eqref{eq nonconvex max}. Therefore, the objective is zero if and only if there
exists a $\p$ with \eqref{eq convex max} and \eqref{eq nonconvex max}.

Even though the model \priceLP\ depends on the parameters 
$\valdelta$ and $\vallambda$, under certain conditions it is 
actually independent of the particular choice of $\valdelta$:

\begin{prop}
Let $\val\Lambda\subseteq\bin^{I\nonconvex}$ and let $(\valdelta,\vallambda)$
and $(\delta',\vallambda)$ be optimal solutions to \masterMICP. 
If the optimal objective value of \priceLP\ is zero,
then the optimal objective of \priceLPof{$\delta',\vallambda$}
is zero and the sets of optimal solutions of both problems coincide.
\end{prop}
\begin{proof}
Let $(\valdelta,\vallambda)$ and $(\delta',\vallambda)$
be optimal solutions to \masterMICP, then $\valdelta$ and $\delta'$
are optimal solutions to \masterMICPof{$\{\vallambda\}$}.
Let the optimal objective value of
\priceLP\ be zero. Then there is a $\p$ with \eqref{eq convex max} 
and \eqref{eq nonconvex max}. Furthermore, $\valdelta$ satisfies
the clearing condition $\sum_{i\in I} f_i(\valdelta_i)=0$, as it is
a feasible solution to \masterMICP. The \welfarethm\ yields that 
$\valdelta$ maximizes the convex program
\begin{align}
\max\quad&\sum_{i\in I} v_i(\delta_i),
  \label{masterCP}\tag{MasterCP)($\val\lambda$}\\
\st\quad
&\sum_{i\in I} f_{i,t}(\delta_i)=0
  &&\forall \hours,\\
&g_{i}(\delta_i)\le 0
  &&\forall i\in I\convex,\\
&A_i\delta_i\le\vallambda_i a_i
  &&\forall i\in I\nonconvex.
\end{align}
Note that \eqref{masterCP} is the convex relaxation of \masterMICPof{$\{\vallambda\}$}.
We know that $\sum_{i\in I}v_i(\valdelta_i)=\sum_{i\in I}v_i(\delta'_i)$,
as both solutions maximize \masterMICPof{$\{\vallambda\}$}. Therefore,
$\delta'$ maximizes \eqref{masterCP}. This time, the \welfarethm\ yields
that there is a $\p$ with \eqref{eq convex max} and \eqref{eq nonconvex max}.
In other words, the objective of \priceLPof{$\delta',\vallambda$} is zero.

Recall that $\valdelta$ and $\delta'$ maximize \eqref{masterCP}. Proposition
\ref{prop optimal multipliers} provides that the set of dual solutions
that satisfy the KKT conditions of \eqref{masterCP} is independent of the
particular primal optimal solution $\valdelta$ or $\delta'$. In other words,
the set of optimal solutions to \priceLP\ is equal to the set of optimal
solutions to \priceLPof{$\delta',\vallambda$}.
\end{proof}
\newcommand\masterCP{\eqref{masterCP}}%
\newcommand\masterCPof[1]{\hyperref[masterCP]{\textnormal{(MasterCP)(#1)}}}%

The next paragraph explains the meaning of the set $\Lambda$. Recall that
we want to solve \rejectionMPEC\ for a mixed integer auction. Theorem 
\ref{thm max or reject} provides the following reformulation:
\begin{align}
\max\quad
&\sum_{i\in I}v_i(\delta_i),
\label{mainMPEC1}\tag{MainMPEC$'$}\\
\st\quad
&\sum_{i\in I}f_{i,t}(\delta_i)=0
    &&\forall\hours,
    \label{mainMPEC1 eq 1}\\
&\delta_i\in\arg\max\{v_i(\delta_i') - {\p}\tp f_i(\delta_i')
\mid g_{i}(\delta'_i)\le 0\}
    &&\forall i\in I\convex,
    \label{mainMPEC1 eq 2}\\
&\delta_i\in\arg\max\{v_i(\delta_i') - {\p}\tp f_i(\delta_i')
\mid A_i\delta'_i\le\lambda_i a_i \}
    &&\forall i\in I\nonconvex,
    \label{mainMPEC1 eq 3}\\
&\delta_i\in\R^{n_i-z_i}\times\Z^{z_i}
  &&\forall i\in I\nonconvex,\\
&\lambda\in\bin^{I\nonconvex}.
\end{align}
Recall that the variables of this model are $\delta$, $\lambda$, and \p.
The fundamental welfare theorem provides that there exists a \p\ such that
$\delta$ satisfies \eqref{mainMPEC1 eq 1}-\eqref{mainMPEC1 eq 3} if
and only if $\delta$ maximizes \masterCPof{$\lambda$}. In this respect,
the previous model is equivalent to:
\begin{align}
\max\quad
&\sum_{i\in I}v_i(\delta_i),
\label{mainMPEC2}\tag{MainMPEC$''$}\\
\st\quad
&\delta\text{ maximizes }\masterCPof{$\lambda$},
    \label{mainMPEC2 eq 1}\\
&\delta_i\in\R^{n_i-z_i}\times\Z^{z_i}
  &&\forall i\in I\nonconvex,
    \label{mainMPEC2 eq 2}\\
&\lambda\in\bin^{I\nonconvex}.
\end{align}
Observe that the $\lambda$-part of a feasible solution to \eqref{mainMPEC2}
is in the set
\begin{align*}
 \Lambda^\circ
 :=&\ \{\lambda\in\bin^{I\nonconvex}
    \mid\text{\masterCPof{$\lambda$} has an optimal solution that is mixed integral}\}\\
 =&\ \{\lambda\in\bin^{I\nonconvex}
    \mid\text{there is a $\delta$ that maximizes
    \masterCPof{$\lambda$} and satisfies \eqref{mainMPEC2 eq 2}}\}.
\end{align*}
The feasible region of \eqref{mainMPEC2} remains unchanged if
we restrict the variable $\lambda$ to the set $\Lambda^\circ$, i.e., 
replace $\lambda\in\bin^{I\nonconvex}$ by $\lambda\in\Lambda^\circ$.
The fundamental welfare theorem yields that this also applies to
\eqref{mainMPEC1}. The following two lemmas will allow us to transform
the bilevel program into a program with just one level.
\begin{lem}
Let $f:\Rn\rightarrow\R$ and $X,Y\subseteq\Rn$.
If $\max\{f(x)\mid x\in X\}$ has an optimal solution $\val x$ with $\val x\in Y$, then
\begin{align*}
\left(\matr{
  \arg\max&f(x)\hfill\\
  \hfill\st&x\in X \cap Y
}\right)
=
\left(\matr{
  \arg\max&f(x)\hfill\\
  \hfill\st&x\in X
}\right) \cap Y.
\end{align*}
\end{lem}
\begin{lem}
Let $f:\Rn\rightarrow\R$, $Y\subseteq\Rm$, and for all $y\in Y$ let $X(y)\subseteq \Rn$.
If for all $y\in Y$ the program $\max\{f(x)\mid x\in X(y)\}$ has an optimal solution, then
\begin{align*}
\left(\matr{
   \arg\max&f(x)\hfill\\
   \hfill\st&\matr{
             x\in\arg\max&f(x)\hfill\\
                \hfill\st&x\in X(y)
            }\\
        &y\in Y\hfill
}\right)
=\left(\matr{
  \arg\max &f(x)\hfill\\
  \hfill\st&x\in X(y)\hfill\\
         &y\in Y\hfill
}\right).
\end{align*}
\end{lem}
The program \eqref{mainMPEC2} is equivalent to \masterMICPof{$\Lambda^\circ$}
\begin{align}
\max\quad&\sum_{i\in I} v_i(\delta_i),
  \tag{MasterMICP)($\Lambda^\circ$}\\
\st\quad
&\sum_{i\in I} f_{i,t}(\delta_i)=0
  &&\forall \hours,\\
&g_{i}(\delta_i)\le 0
  &&\forall i\in I\convex,\\
&A_i\delta_i\le\lambda_i a_i
  &&\forall i\in I\nonconvex,\\
&\delta_i\in\R^{n_i-z_i}\times\Z^{z_i}
  &&\forall i\in I\nonconvex,\\
&\lambda\in\Lambda^\circ.
  &&
\end{align}
Putting all together we obtain that \rejectionMPEC\ is equivalent to
\masterMICPof{$\Lambda^\circ$}. In other words, there is a $\vallambda$
such that $(\valdelta,\vallambda)$ maximizes \masterMICPof{$\Lambda^\circ$}
if and only if there is a $\p$ such that $(\valdelta,\p)$ maximizes
\rejectionMPEC.

\begin{prop}
Let $\val\Lambda\subseteq\bin^{I\nonconvex}$. If $(\valdelta,\vallambda)$ maximizes
\masterMICP\ and the optimal objective value of \priceLP\ is non-zero
then \masterCP\ has no optimal solution that is mixed integral, i.e.,
$\vallambda\notin\Lambda^\circ$.
\end{prop}
\begin{proof}
The optimal objective of \priceLP\ is zero if and only if there exist dual 
variables that satisfy the KKT conditions. It follows that $\valdelta$ is not 
an optimal solution to \masterCP. Suppose that there is an optimal solution 
$\delta'$ to \masterCP\ that is mixed integral, then 
$\sum_{i\in I}v_i(\valdelta) < \sum_{i\in I}v_i(\delta')$.
The solution $\delta'$ is also feasible to \masterMICPof{$\{\vallambda\}$}
and $\valdelta$ maximizes \masterMICPof{$\{\vallambda\}$}, thus 
$\sum_{i\in I}v_i(\delta') \le \sum_{i\in I}v_i(\valdelta)$.
This is a contradiction.
\end{proof}

Now we come to the next step of the algorithm. Recall that
$(\valdelta,\vallambda)$ is an optimal solution to \masterMICP.
Let $(\p,\mu)$ be an optimal solution to \priceLP. 
If the objective value is zero, then $(\valdelta,\p)$ is an
optimal solution to \rejectionMPEC\ and we are finished.
Otherwise we notice that $\vallambda\notin\Lambda^\circ$, thus
our set $\val\Lambda$ is to big. The previous proposition shows
that we do not cut of any feasible solution to \masterMICP\ if we
remove $\vallambda$ from $\val\Lambda$. Therefore, we set
$\val\Lambda\gets\val\Lambda\setminus\{\vallambda\}$, go back to
the first step and use the modified set $\val\Lambda$. If there 
exists a competitive equilibrium the algorithm finds it in the
first iteration. Otherwise it finds a \emph{bid selection}
$\vallambda$, such that all \emph{selected} mixed integer bids 
(i.e., bids with $\vallambda_i=1$) and all convex bids are profit 
maximizing. All \emph{rejected} mixed integer bids (i.e., bis with 
$\vallambda_i=0$) are not executed at all ($\valdelta_i=0$).
The algorithm terminates after a finite number of steps, because
$\val\Lambda$ is finite and the size of the set decreases in each step.
The procedure is summarized in Algorithm \ref{alg exact bid cut}.

\begin{algorithm}[thp]
\begin{algorithmic}
\medskip
\REQUIRE Instance of a mixed integer auction
\ENSURE Optimal solution to \rejectionMPEC
\STATE
\STATE $\val\Lambda \gets \bin^{I\nonconvex}$
\STATE $done \gets \FALSE$

\WHILE{$\neg done$}
    \STATE $(\valdelta,\vallambda) \gets$ solve \masterMICP
    \STATE $(\valp,\val\mu) \gets$ solve parameterized model \priceLP
\medskip
    \IF {$\left(
           \sum_{i\in I\nonconvex} {\val\mu_i}\tp(A_i\valdelta_i-\vallambda_i a_i) = 0
          \right)$,}
        \STATE {\it// there exist prices such that no bid incurs a loss}
        \STATE $done\gets\TRUE$
    \ELSE
        \STATE {\it// current solution is infeasible, reject it}
        \STATE $\val\Lambda\gets\val\Lambda\setminus\{\vallambda\}$
    \ENDIF 
\medskip
\ENDWHILE
\medskip
\STATE {\bf return} $(\valdelta,\valp)$
\medskip
\end{algorithmic}
\caption{Exact algorithm for mixed integer auctions}
\label{alg exact bid cut}
\end{algorithm}

The step $\val\Lambda\gets\val\Lambda\setminus\{\vallambda\}$ can be
implemented by adding the following cut to the model:
\[
    \sum_{i\in I\nonconvex:\vallambda_i=0}\lambda_i
    +\sum_{i\in I\nonconvex:\vallambda_i=1}(1-\lambda_i)\ge 1.
\]
The exact algorithm should be combined with a heuristic. We obtain a fast 
heuristic if we use a more aggressive cut instead of the previous one.
Let $L$ be the set of mixed integer bids that incur a loss in the current
iteration, then a heuristic cut is given by
\begin{align*}
    \sum_{i\in L}\lambda_i\le |L|-1,&&\text{where}&&
    L:=\{i\in I\nonconvex\mid {\val\mu_i}\tp(A_i\valdelta_i-\vallambda_i a_i) < 0\}.
\end{align*}

\subsection{Model B: Non-Negative Surplus}
\label{sec non-negative surplus}

In Model A, non-convex bids are either surplus maximizing or rejected.
This is a very mild relaxation of a competitive equilibrium.
However, if we consider non-trivial mixed integer bids, the model may imply certain diseconomies.

\begin{example}\label{ex uc order}
A participant $i$ has a production plant. If the plant is used, then on 
the one hand start-up costs of $s=\EUR{30}$ arise and on the other 
hand marginal costs of $m=\EUR{10}$/unit arise for each produced unit.
The variable $\delta_{i,1}\in\bin$ models whether the 
plant is used or not and $\delta_{i,2}\in[0,u]$ models the 
produced quantity which is limited to $u=50$ units.
The price parameterized optimization problem is as follows:
\begin{align*}
\left(\matr{
\hfill\max\quad&v_i(\delta_i)-\p\tp f_i(\delta_i)\hfill\\
\hfill\st\quad&\delta_i\in D_i\hfill
}\right)=
\left(\matr{
\hfill\max\quad&-(s,m)\tp\delta_{i}-\p\tp(-1)\delta_{i,2}\hfill\\
\hfill\st\quad&\delta_{i,1}\in\bin\hfill\\
        &0\le\delta_{i,2}\le u\delta_{i,1}\hfill
}\right)\tag{SellBid)($\p$}\label{ex uc sell bid},
\end{align*}
where $\p\in\R$ is the exogenously given price.
Recall that negative quantities denote production. For this
reason the objective contains the additional minus signs.
The producer receives $-{\p}\tp f_i(\delta_i)=\p\delta_{i,2}$ Euro for the production of
$\delta_{i,2}$ units. The production costs are covered if 
\begin{align}
 v_i(\delta_i) - {\p}\tp f_i(\delta_i)\ge0.
\qquad\Leftrightarrow\qquad
\p\delta_{i,2}-s\delta_{i,1}-m\delta_{i,2}\ge 0\label{ex uc order eq1}
\end{align}
If $\delta_{i,2}>0$, this boils down to $\p\ge m+s/\delta_{i,2}=10+30/\delta_{i,2}$.
If $\delta_{i,2}=0$, the inequality yields $\delta_{i,1}=0$.
The blue area and the blue line in Figure \ref{fig:uc:bid:example} depict the 
price-quantity combinations $(\p,\delta_{i,2})$ where inequality \eqref{ex uc order eq1}
holds.
In Model A the decision variable $\delta_i$ of a mixed integer bid either maximizes
the individual optimization problem or it is zero, that is, 
\begin{align}
\delta_i\in\arg\max\{v_i(\delta_i') - {\p}\tp f_i(\delta_i')
\mid \delta_i'\in D_i\}\cup\{0\}.\label{ex uc order eq2}
\end{align}
The blue lines in Figure \ref{fig:uc:bid:example:block} depict the price-quantity
combinations that satisfy equation this equation. We can see that the blue area
in this Figure is significantly smaller than in the previous Figure.
\begin{figure}[tbp]
\centering\includegraphics[width=0.7\textwidth]{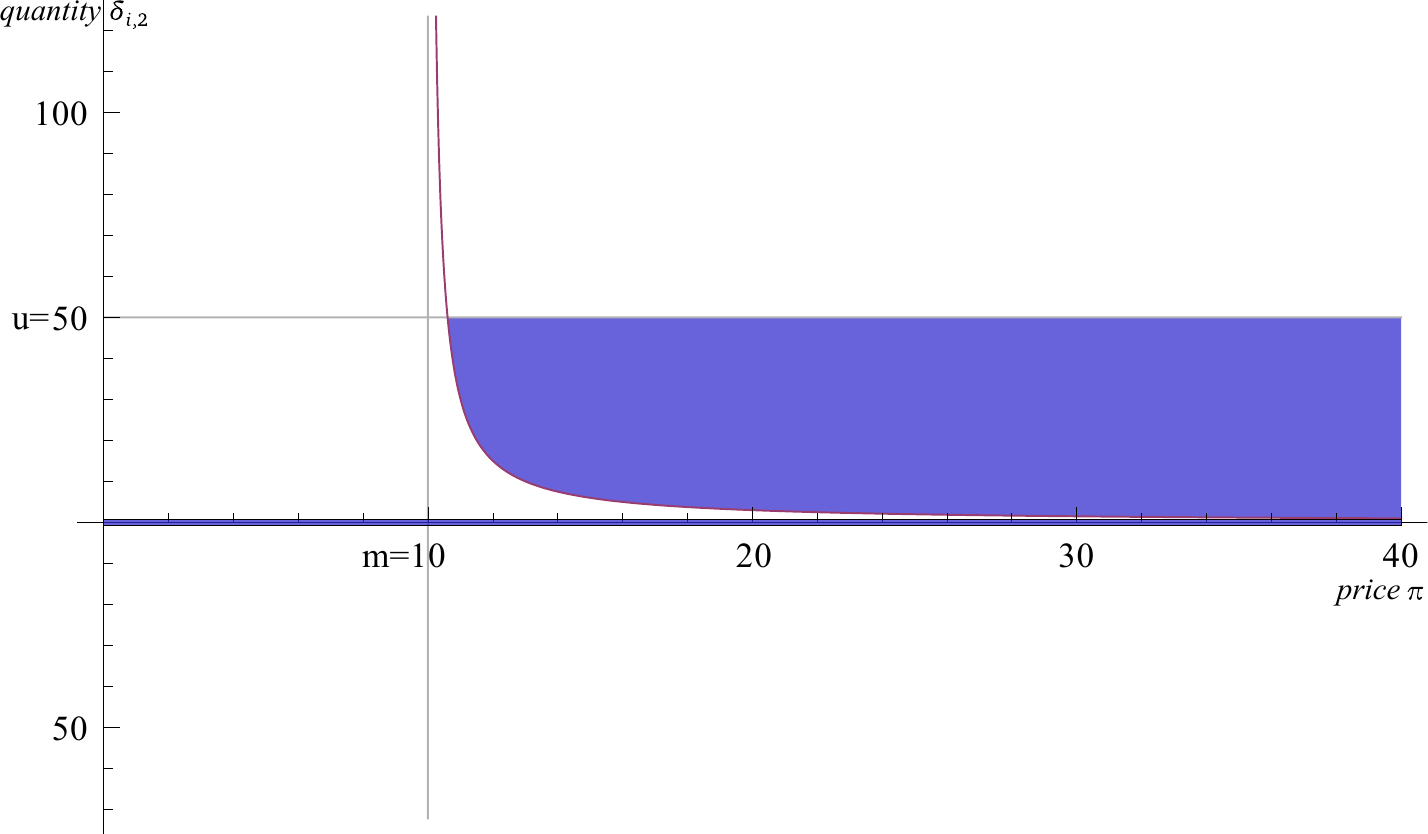}
\caption{thick line and marked area: price-quantity combinations with a non-negative surplus.}
\label{fig:uc:bid:example}
\end{figure}
\begin{figure}[tbp]
\centering\includegraphics[width=0.7\textwidth]{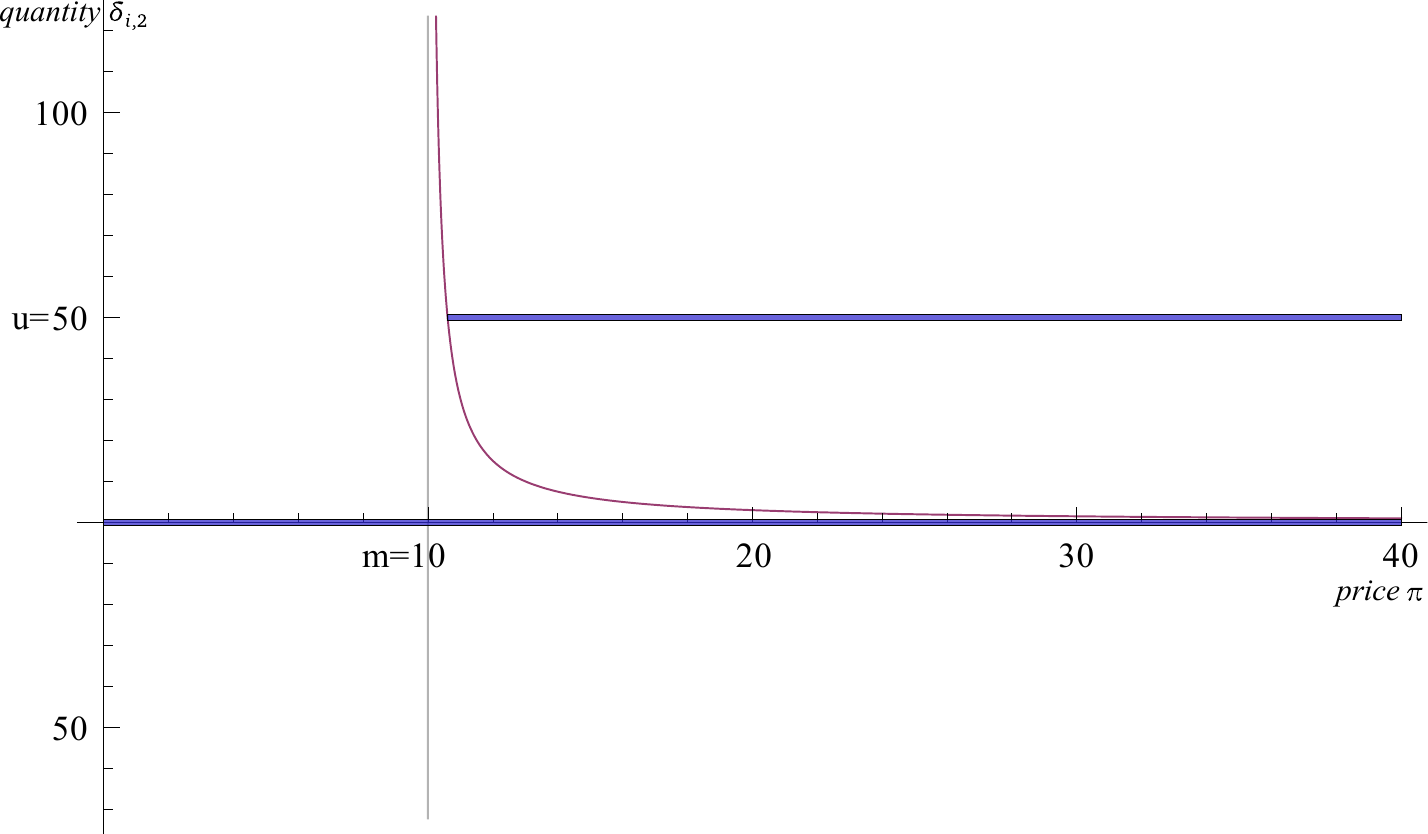}
\caption{thick lines:
quantities are either zero or maximize \eqref{ex uc sell bid}.}
\label{fig:uc:bid:example:block}
\end{figure}

Now assume that there is a convex demand bid $j$. The marginal willingness to pay
amounts to $\EUR{20}$/unit and the maximal demanded quantity amounts to $40$ units.
The individual optimization problem is as follows:
\begin{align*}
\left(\matr{
\hfill\max\quad&v_j(\delta_j)-\p\tp f_j(\delta_j)\hfill\\
\hfill\st\quad&\delta_j\in D_j\hfill
}\right)=
\left(\matr{
\hfill\max\quad&20\delta_j-\p\tp\delta_j\hfill\\
\hfill\st\quad&\delta_j\in[0,40]\hfill
}\right)\tag{BuyBid)($\p$}\label{ex uc buy bid},
\end{align*}
In Model A the sell bid either produces nothing or 50 units (cf. Figure 
\ref{fig:uc:bid:example:block}). Therefore, the buy bid cannot be executed,
because buys at most 40 units. Model A ends up with an economic surplus of
zero Euro. If we replace the equation \eqref{ex uc order eq2} by the less
restrictive constraint \eqref{ex uc order eq1}, the following solution becomes
feasible: $\delta_i=(1,40)$, $\delta_j=40$, and $\p=15$. The solution
generates an economic surplus of $20\cdot40-30-10\cdot40=370$ Euro. The sell bid $i$
has a non-negative surplus of $15\cdot40-30-10\cdot40=170$ Euro and the decision 
variable $\delta_i$ is feasible for \eqref{ex uc sell bid}. The buy bid $j$ has 
a surplus of $20\cdot40-15\cdot40=200$ Euro and the decision variable $\delta_j$ 
maximizes \eqref{ex uc buy bid}.
\end{example}

The example shows that it may be disadvantageous to enforce constraint
\eqref{ex uc order eq2}. This is done in Model A, which computes the trivial
zero solution in the previous example. No bid is executed, even though there
exists a solution where both participants have a strictly positive surplus.
This is not desirable. To address this issue, one 
can consider the following relaxation of \rejectionMPEC. 
\begin{align}
\max\quad
&\sum_{i\in I}v_i(\delta_i),\label{model b eq0}\\
\st\quad
&\sum_{i\in I}f_{i,t}(\delta_i)=0
    &&\forall\hours,\\
&\delta_i\in\arg\max\{v_i(\delta_i') - {\p}\tp f_i(\delta_i')
\mid \delta_i'\in D_i\}
    &&\forall i\in I\convex,\label{model b eq2}\\
&\delta_i\in D_i
\quad\und\quad
v_i(\delta_i) - {\p}\tp f_i(\delta_i) \ge 0
    &&\forall i\in I\nonconvex.\label{model b eq3}
\end{align}
\newcommand\hybridMPEC{\eqref{model b eq0}-\eqref{model b eq3}}
This model ensures that the decision variables of each convex bid maximize the
individual surplus and the decision variables of each non-convex bid are
feasible and the associated surplus is non-negative. Similar to constraint
\eqref{ex uc order eq2} in model \rejectionMPEC, constraint \eqref{model b eq2}
may lead to unfavorable solutions. Model \eqref{model b eq0}-\eqref{model b eq3}
may compute the trivial zero solution, even if there exists a solution 
where all participants have a strictly positive surplus.
These considerations lead to a further relaxation of \rejectionMPEC.
\begin{align*}
\max\quad
&\sum_{i\in I}v_i(\delta_i),\tag{NoLoss}\label{non-negative}\\
\st\quad
&\sum_{i\in I}f_{i,t}(\delta_i)=0
    &&\forall\hours,\\
&\delta_i\in D_i
\quad\und\quad
v_i(\delta_i) - {\p}\tp f_i(\delta_i) \ge 0
    &&\forall i\in I.
\end{align*}
\newcommand\nonnegativeMPEC{\eqref{non-negative}}%
The model does not distinguish between convex and non-convex bids. It 
maximizes the economic surplus subject to the clearing condition,
the feasibility of the decision variables, and the guarantee that
no participant incurs a loss. Recall that $\delta$ and $\p$ are the
variables of the model. Therefore, the non-negative surplus constraint
is non-convex, thus it is difficult to handle.

In the following paragraph we show that the diseconomies described above do not
arise in the model \nonnegativeMPEC. Therefore, we introduce an efficiency term
that characterizes economically desirable solutions and takes 
the strict linear prices into account.
\begin{defn}
A tuple $(\valdelta,\valp)$ is an \emph{efficient} solution to
the multicriteria optimization problem
\begin{align*}
\max\quad
&\vect{v_1(\delta_1)-\p\tp f_1(\delta_1)\\
         \vdots\\
         v_{|I|}(\delta_{|I|})-\p\tp f_{|I|}(\delta_{|I|})}
    \tag{Multi}\label{Multi}\\
\st\quad
&\sum_{i\in I}f_{i,t}(\delta_i)=0
    &&\forall\hours,\\
&\delta_i\in D_i
    &&\forall i\in I,\\
&\p_t\in\R
    &&\forall\hours
\end{align*}
if there is no other feasible solution $(\delta,\p)$ to \eqref{Multi} with
\begin{align}
 v_i(\delta_i)-\p\tp f_i(\delta_i)
  &\ge v_i(\valdelta_i)-{\valp}\tp f_i(\valdelta_i)
  &&\text{for all }i\in I\und
  \label{eq efficient 1}\\
 v_i(\delta_i)-\p\tp f_i(\delta_i)
  &> v_i(\valdelta_i)-{\valp}\tp f_i(\valdelta_i)
  &&\text{for some }i\in I
  \label{eq efficient 2}.
\end{align}
In other words, $(\valdelta,\valp)$ is an \emph{efficient} solution to
\eqref{Multi} if the surplus of a participant cannot be increased without
decreasing the surplus of another participant.
\end{defn}

\begin{prop}
An optimal solution to \nonnegativeMPEC\ is an efficient solution to \eqref{Multi}.
\end{prop}
\begin{proof}
Let $(\valdelta,\valp)$ be an optimal solution to \nonnegativeMPEC.
At this solution the surplus of each participant $i\in I$ is non-negative:
$v_i(\valdelta_i)-{\valp}\tp f_i(\valdelta_i)\ge0$.
Assume that there is a feasible solution $(\delta,\p)$ to \eqref{Multi} with
\eqref{eq efficient 1} and \eqref{eq efficient 2}.
Equation \eqref{eq efficient 1} yields
\[
v_i(\delta_i)-\p\tp f_i(\delta_i)
  \ge v_i(\valdelta_i)-{\valp}\tp f_i(\valdelta_i)\ge0
  \qquad\text{for all }i\in I.
\]
In other words, $(\delta,\p)$ is also a feasible solution to \nonnegativeMPEC.
We will now compare the objective values of both solutions.
\[
\sum_{i\in I}v_i(\delta_i)=
  \sum_{i\in I}\left(v_i(\delta_i)-\p\tp f_i(\delta_i)\right)
  \explain{>}{\eqref{eq efficient 1}\und\eqref{eq efficient 2}}
  \sum_{i\in I}\left(v_i(\valdelta_i)-{\valp}\tp f_i(\valdelta_i)\right)=
  \sum_{i\in I}v_i(\valdelta_i).
\]
This is a contradiction, because $(\valdelta,\valp)$ maximizes \nonnegativeMPEC.
The objective value of $(\delta,\p)$ cannot be strictly greater than
the objective value of $(\valdelta,\valp)$.
\end{proof}

An optimal solution to \nonnegativeMPEC\ is an allocation that maximizes
the economic surplus such that there exists a strict linear pricing 
schedule where no participant incurs a loss. It is not possible to
increase the surplus of one participant without decreasing the surplus of
another participant. In particular any allocation with a higher economic 
surplus requires the use of another pricing schedule, because otherwise
at least one participant incurs a loss. Due to these properties
we call an optimal solution to \nonnegativeMPEC\ a \emph{strict linear 
price equilibrium}.

\subsection{Read-World Application}\label{sec real-world application}
The previous models and algorithms are motivated by a 
real-world application. In our previous paper (\cite{mmp2013}) we modeled the 
European day-ahead electricity auction and developed an exact algorithm and a
heuristic for this specific auction. The electricity auction model is 
similar to model \hybridMPEC.
Even though the algorithms in \cite{mmp2013} are designed to solve the specific 
electricity auction problem, they can be generalized. The generalized algorithms 
can solve Model A for arbitrary mixed integer auctions (cf. Algorithm 
\ref{alg exact bid cut} in Section \ref{sec maximization or rejection}).
Note that in general Model A differs from \hybridMPEC, and 
therefore, does not reflect all requirements of the electricity auction.
However, the specialized algorithms in \cite{mmp2013} are similar to the
ones in this paper. They are applicable to large scale real-world electricity
auctions and produce high quality solutions in a short time.

\section{Summary and Outlook}

The article presents two relaxations of a competitive equilibrium in general auctions.
For the first relaxation \rejectionMPEC\, we constructed an exact algorithm and
a heuristic. Both algorithms are applicable to mixed integer auctions.
Even though the first relaxation has desirable algorithmic properties, we showed that
an optimal solution might not be efficient from an economical point of view: 
there exist situations, where all participants can be made better off. Such an 
improvement can be achieved if we replace all optimality conditions by 
the condition that no participant may incur a loss. This leads us to the second
relaxation \nonnegativeMPEC, which maximizes the economic surplus
such that no participant incurs a loss. A solution to this model turns out to be
efficient: no participant can be made better off without making another one worse off.
The second relaxation is economically preferable, but the algorithms for the first
relaxation, only provide heuristic solutions to the second relaxation.
Further research could be concerned about algorithms that can solve large scale
mixed integer auction instances of the second relaxation.

\section*{Acknowledgements}
We want to thank Alexander Martin and Deutsche Börse Systems for supporting this
work. We also want to thank our colleagues for the discussions about the topic.

\begin{appendix}
\setcounter{equation}{0}
\renewcommand{\theequation}{A.\arabic{equation}}
\renewcommand{\thesection}{\appendixname\ \Alph{section}.}

\section{Convex Optimization}

In this section we shortly introduce some basic definitions and results
concerning convex optimization. Consider the convex optimization problem
\begin{align}
\max\quad&v(x),\label{convex prob eq 1}\\
\st\quad&f_i(x)=0&&i=1,\dots,k,\label{convex prob eq 2}\\
	&g_i(x)\leq 0&&i=1,\dots,m,\label{convex prob eq 3}
\end{align}%
\newcommand\convexprob{\eqref{convex prob eq 1}-\eqref{convex prob eq 3}}%
where $v,f_i,g_i:\R^n\rightarrow\R$ are differentiable functions, $v$ concave,
$g_i$ convex, and $f_i$ affine. To describe optimal solutions to this problem
it is convenient to assume that the weak Slater assumption holds.
\cite{hiriart1993convex} provide the following definition:

\begin{defn}
The constraints \eqref{convex prob eq 2}-\eqref{convex prob eq 3} satisfy the
\emph{weak Slater assumption} if there is a feasible point at which all the
non-affine constraints are strictly satisfied, i.e., there exists $\hat x\in\R^n$ with
\begin{align*}
f_i(\hat x)& = 0&&i=1,\dots,k,\\
g_i(\hat x)&\le0&&i=1,\dots,m\text{ and $g_i$ is affine},\\
g_i(\hat x)& < 0&&i=1,\dots,m\text{ and $g_i$ is non-affine}.
\end{align*}
\end{defn}

With this assumption the \emph{Karush-Kuhn-Tucker} conditions, short KKT conditions,
are necessary and sufficient to describe an optimal solution to our convex problem:

\begin{thm}[KKT for differentiable convex problems]
\label{satz:kkt:cont:konvex}
Let $v,f_i,g_i:\R^n\rightarrow\R$ be differentiable functions, $v$ concave,
$g_i$ convex, and $f_i$ affine. If for $\bar x,\bar\p,\bar\lambda$ the \emph{KKT conditions}
\begin{align}
f_i(\bar x)&=0&&i=1,\dots,k,\\
g_i(\bar x)&\leq0&&i=1,\dots,m,\\
\sum_{i=1}^{k}\bar\p_i \frac{\partial f_i}{\partial x_j}(\bar x)
+\sum_{i=1}^{m}\bar\lambda_i \frac{\partial g_i}{\partial x_j}(\bar x)
&=\frac{\partial v}{\partial x_j}(\bar x)&&j=1,\dots,n,\label{eq cp dual 1}\\
\bar\lambda_i&\geq0&&i=1,\dots,m,\label{eq cp dual 2}\\
\bar\lambda_i g_i(\bar x)&=0&&i=1,\dots,m,\label{eq cp dual 3}
\end{align}
hold, then $\bar x$ is primal optimal and $(\bar\p,\bar\lambda)$ dual optimal for
\convexprob. If $\bar x$ is primal optimal and the weak Slater assumption holds, 
then there exists $(\bar\p,\bar\lambda)$ such that the KKT conditions are satisfied.
\end{thm}
\begin{proof}
Confer Theorem VII.2.1.4 and VII.2.2.5 in \cite{hiriart1993convex}
or Chapter 5.5.3 in \cite{boyd04}.
\end{proof}

\begin{prop}\label{prop optimal multipliers}
Let $v,f_i,g_i:\R^n\rightarrow\R$ be differentiable functions, $v$ concave,
$g_i$ convex, and $f_i$ affine.
Let $\bar x$ and $\bar x'$ be optimal solutions to \convexprob\ and let 
$M(\bar x):=\{(\bar\p,\bar\lambda)\mid \eqref{eq cp dual 1}-\eqref{eq cp dual 3}\}$.
Then $M(\bar x)=M(\bar x')$.
\end{prop}
\begin{proof}
Confer Proposition VII.3.3.1 in \cite{hiriart1993convex}.
\end{proof}

\end{appendix}

\urlstyle{same}
\renewcommand\path[1]{#1}
\bibliographystyle{model5-names}
\bibliography{literatur}

\end{document}